\documentclass[preprint,12pt]{elsarticle}

\usepackage{amsmath}
\usepackage{amsthm}
\usepackage{amssymb}
\newcommand\DN{\newcommand}
\newcommand\DR{\renewcommand}

\theoremstyle{plain}
\newtheorem{thm}{Theorem}[section]  

\newtheorem{lem}[thm]{Lemma}

\theoremstyle{definition}
\newtheorem{dfn}{Definition}[section]

\theoremstyle{remark}
\newtheorem{rem}{Remark}[section]

\DN\lref[1]{Lemma~\ref{#1}}
\DN\tref[1]{Theorem~\ref{#1}}
\DN\pref[1]{Proposition~\ref{#1}}
\DN\sref[1]{Section~\ref{#1}}
\DN\dref[1]{Definition~\ref{#1}}
\DN\rref[1]{Remark~\ref{#1}} 
\DN\correff[1]{Corollary~\ref{#1}}
\DN\eref[1]{Example~\ref{#1}}

\DN\map[3]{#1\!:\!#2\!\to\!#3}
\DN\limi[1]{\lim_{#1\to\infty}} 	
\DN\limz[1]{\lim_{#1\to 0}} 	
\DN\lsupi[1]{\limsup_{#1\to\infty}} 	
\DN\linfi[1]{\liminf_{#1\to\infty}} 	
\DN\lsupz[1]{\limsup_{#1\to 0}} 	
\DN\linfz[1]{\liminf_{#1\to 0}} 	
\DN\PD[2]{\frac{\partial#1}{\partial#2}}

\numberwithin{equation}{section}

\newcounter{Const} 
\DN\Ct{\refstepcounter{Const}c_{\theConst}}%\label{#1}
\DN\cref[1]{c_{\ref{#1}}}

\DN\As[1]{\textbf{(#1)}}

\journal{Stochastic Processes and Their Applications}

\DN\half{\frac{1}{2}}
\DN\Rd{\R ^d}
\DN\R{\mathbb{R}}\DN\N{\mathbb{N}}
\DN\Q{\mathbb{Q}}\DN{\Z}{\mathbb{Z}}

\DN\ot{\otimes} 
\DN\ts{\times }

\DN\limsupi[1]{\limsup_{#1\to\infty}} 	
\DN\liminfi[1]{\liminf_{#1\to\infty}}

\DN\mr{m_r }

\DN\iii{\mathrm{i}}
\DN\q{\rrr }
\DN\rrr{{r}}%[]
\DN\ww{\rr } %[]
\DN\www{\rr (x)}
\DN\rr{{w}}
\DN\rrN{\rr }
\DN\rrNs{\rrN _{s}}
\DN\rrbar{\bar{\rr }}

\DN\ulab{\mathfrak{u} }
\DN\lab{\mathfrak{l} }
\DN\hhh{\mathfrak{h}}

\DN\g{g} 
\DN\gN{\g }
\DN\gNrs{\gN _{rs}}
\DN\gNs{\gN _{s}}

\DN\ggN{\mathfrak{g}}
\DN\ggNs{\ggN _{s}}

\DN\ggrs{\mathsf{g}_{rs}} 
\DN\ggst{\mathsf{g}_{st}} 
\DN\ggrt{\mathsf{g}_{rt}} 

\DN\ggr{\mathsf{g}_{r}} \DN\ggtilder{\ggtilde_{r}}
\DN\ggs{\mathsf{g}_{s}}\DN\ggtildes{\ggtilde_{s}}
\DN\ggtilde{\tilde{\mathsf{g}}}
\DN\ggtilders{\tilde{\mathsf{g}}_{rs}}

\DN\hh{\mathfrak{h}}
\DN\hsym{h _{\mathrm{sym}}}

\DN\Y{(\mathsf{y}=\sum_{i=1}^{\n -1}\delta_{y_i})}

\DN\uu{u}
\DN\uN{\uu ^{\n }}
\DN\vN{v^{\n }}
\DN\vNsinfty{\int_{s\le |x-y|}\vN (x,y)dy }
\DN\vNs{\int_{|x-y|< s}\vN (x,y)dy }

\DN\vvv{v}
\DN\vsinfty{\int_{s \le |x-y|}\vvv (x,y)dy }
\DN\vs{\int_{|x-y|< s}\vvv (x,y)dy }

\DN\supN{\sup_{ \n \in\N}}
\DN\xyi{|x-y_i|}
\DN\yi{|y_i|}

\DN\mubN{\mu ^{\n }_{\alpha }}
\DN\mubNx{\mu ^{\n }_{\alpha ,x}}

\DN\muz{\mu _{0}}
\DN\dxmu{dx \ts \mu }
\DN\muk{\mu ^{[k]}}
\DN\muone{\mu ^{[1]}}

\DN\muN{\mu ^{\n }}
\DN\muNx{\mu ^{\n }_{x}}
\DN\muNzero{\mu ^{\n }_{0}}
\DN\muNone{\mu ^{\n ,[1]}}

\DN\muNstar{\mu ^{\n *}}

\DN\rNone{\rho ^{\n ,1}}
\DN\rNtwo{\rho ^{\n ,2}}
\DN\rNnk{\rho ^{\n ,n+k}}

\DN\rhobar{\bar{\rho }}
\DN\rb{\rho _{\alpha }}
\DN\rbn{\rb ^n}
\DN\rbm{\rb ^m}

\DN\rbone{\rho _{\alpha }^1}
\DN\rbx{\rho _{1 }^n}
\DN\rby{\rho _{2 }^n}
\DN\rbz{\rho _{4 }^n}

\DN\rbNone{\rb ^{\n ,1}}
\DN\rbNtwo{\rb ^{\n ,2}}
\DN\rbNxone{\rho _{\alpha ,x}^{\n ,1}}
\DN\rbNxtwo{\rho _{\alpha ,x}^{\n ,2}}
\DN\rbNxn{\rho _{\alpha ,x}^{\n ,n}}
\DN\rN{\rho ^{\n ,n}}
\DN\rbN{\rb ^{\n ,n}}
\DN\rbNN{\rb ^{\n ,\n }}
\DN\rbNx{\rho _{1}^{\n ,n}}
\DN\rbNy{\rho _{2}^{\n ,n}}
\DN\rbNz{\rho _{4}^{\n ,n}}

\DN\rg{\rho _{\mathrm{bes}}}
\DN\rgx{\rho _{\mathrm{bes},x}}
\DN\rgn{\rho _{\mathrm{bes}}^n}
\DN\rgNone{\rho ^{\n ,1}_{\mathrm{bes}}}
\DN\rgNonex{\rho ^{\n ,1}_{\mathrm{bes},x}}
\DN\rbNg{\rho _{\n ,g, \alpha }^{n}}

\DN\Ka{\mathsf{K}_{\alpha }}
\DN\Kan{\Ka ^{\n }}
\DN\Kanx{\mathsf{K}_{\alpha , x}^{\n }}
\DN\Ma{M_{\alpha }^{\n }}
\DN\mub{\mu _{\alpha }}

\DN\SSS{\mathsf{S}}
\DN\SSSS{\mathbf{S}}
\DN\SSSr{\SSS _{r}}
\DN\SSSrm{\SSS _{r,m}}
\DN\SSSrNk{\SSS _{r,\n -k}}
\DN\SSSg{\SSS _{\mathrm{bes}}}
\DN\SSSSg{\SSSS _{\mathrm{bes}}}
\DN\SSSdys{\SSS _{\mathrm{be}}}
\DN\SSSSdys{\SSSS _{\mathrm{be}}}

\DN\SkS{\Sk \ts \SSS }
\DN\SSSksingle{\SSSsi ^{k}}
\DN\SSSsi{\SSS _{\mathrm{s.i.}}}
\DN\SSSone{\SSS ^{\mathbf{1}}}
\DN\SoneSSS{\SS \ts \SSS }

\DN\Akr{\mathsf{A}^{k}_{r}}
\DN\Akrr{\mathsf{A}^{k}_{r+1}}
\DN\AkrN{\mathsf{A}^{k}_{r,\n }}

\DR\SS{{S}}
\DN\Sk{\SS ^{k}}
\DN\Srk{\Sr ^{k}}
\DN\Sr{\SS _{r}}
\DN\Ss{\SS _{s}}
\DN\Sq{\SS _{\q }}
\DN\Srr{\SS _{r}^{k}}

\DN\dlog{\mathsf{d}}
\DN\dmu{\dlog ^{\mu }}
\DN\dmuN{\dlog ^{\n }}
\DN\dmuone{\dlog ^{\muone }}

\DN\bbb{\mathsf{b}}
\DN\aaa{\mathsf{a}}
\DN\ssss{(s_i,(s_j)_{j\not=i})}
\DN\sss{\mathsf{s}}
\DN\xsss{(x,\sss )}
\DN\xxxx{(x,(y_j)_{j\in\N })}
\DN\xxxxx{(x,\sum_{j\in\N }\delta _{y_j})}
\DN\CY{C^{\infty} (\Rd )\ot\dY } 

\DR\d{\mathcal{D}} 
\DN\di{\d _{\circ}}
\DN\dik{\di ^{\muk }}
\DN\dak{\d ^{\muk }}

\DN\Lm{L^{2}(\mu )}
\DN\Lmuk{L^{2}(\muk )}
\DN\Llocp{L_{\mathrm{loc}}^{p}}
\DN\Llocq{L_{\mathrm{loc}}^{q}}
\DN\Lloctwo{L_{\mathrm{loc}}^{2}}
\DN\Llocone{L_{\mathrm{loc}}^{1}}

\DN\XY{\Rd \ot \SSS }
\DN\RdT{\Rd \ts \SSS }

\DN\E{\mathcal{E}}
\DN\Ea{\E }
\DN\Eak{\E ^{\muk }}
\DN\Eaone{\E ^{\muone }}
\DN\Eazero{\E ^{\mu }}

\DN\DDD{\mathbb{D}}
\DN\DDDk{\mathbb{D}^{k}}
\DN\DDDzero{\mathbb{D}^{0}}
\DN\DDDr{\DDD _{r}}
\DN\D{\mathbf{D}}

\DN\iRT{\int_{\Rd \ts \SSS }}
\DN\iR{\int_{\Rd }}
\DN\iT{\int_{\SSS }}

\DN\Rdk{(\Rd)^{k}}
\DN\PP{\mathsf{P}}
\DN\Pxt{\PP _{(x,\sss ) }}
\DN\Pmt{\PP _{\sss }}
\DN\PPk{\PP ^{k}}

\DN\Done{\mathbb{D}^{1}}

\DN\kpath{\ulab _{\mathrm{path}}}
\DN\kkpath{\ulab _{\mathrm{k,path}}}

\DN\1{1_{\Sq }}
\DN\2{\limi{s}\limsupi{\n } \sup_{x\in\Sq }}
\DN\4{\limsupi{\n }\sup_{x\in\Sq }}

\DN\3{\sup_{ \n \in\N }\sup_{x\in\Sq }}
\DN\9{(\Eak , \dak ,\Lmuk )}

\DN\nN{N_{\n }}
\DN\n{\mathsf{n}} \DN\m{\mathsf{m}}
\DN\mm{\sigma } 

\DN\aaaaa{\noindent {\em Proof. }}
\DN\bbbbb{ \qed  \medskip }

\DN\PhiN{\Phi ^{\n }}
\DN\PsiN{-\beta \log |x-y|}
\DN\mmi{\mathsf{m}_{\infty}^{\n }}
\DN\infN{\inf_{\n\in\N }}
\DN\Ass[1]{\thetag{\textbf{#1}}}
\DN\ZI{[0,\infty)}
\DN\zI{(0,\infty)}
\DN\kan{\mathsf{k}_{\alpha }^{\n }}
\DN\w{\mathsf{w}_{\alpha }}
\DN\Ana{A_{\n , \alpha }}
\begin{document}

\begin{frontmatter}



\title{Infinite-dimensional stochastic differential equations 
related to Bessel random point fields }


\author{Ryuichi Honda, Hirofumi Osada  \\
{\small (Accepted in Stochastic Processes and Their Applications)} }

\address{Faculty of Mathematics, Kyushu University, 
Fukuoka 819-0395, Japan }

\begin{abstract}
We solve the infinite-dimensional stochastic differential equations (ISDEs) describing an infinite number of Brownian particles in $ \mathbb{R}^+$ interacting through the two-dimensional Coulomb potential. The equilibrium states of the associated unlabeled stochastic dynamics are Bessel random point fields. To solve these ISDEs, we calculate the logarithmic derivatives, and  prove that the random point fields are quasi-Gibbsian. 
\end{abstract}

\end{frontmatter}

% \linenumbers


\section{ Introduction }\label{s:1}
The Bessel random point fields $ \mub $ ($ -1 < \alpha < \infty $) 
are probability measures on the configuration space $ \SSS $
over $ \SS = \ZI $,  whose $ n $-point correlation functions $ \rbn $ 
(see \eqref{:20c})
with respect to the Lebesgue measure are given by 
\begin{align} \label{:10a} 
& \rbn (x_1,\ldots,x_n) = \det [\Ka (x_i,x_j)]_{1 \le i,j \le n}
.\end{align}
Here, $ \Ka (x,y)$ is a continuous function called 
the Bessel kernel defined with 
the Bessel function $ J_{\alpha }$ of order $ \alpha $ such that 
for $ x\not=y $ 
\begin{align}\label{:10b}
 \Ka (x,y) & = 
 \frac{J_{\alpha } (\sqrt{x}) \sqrt{y} J_{\alpha }' (\sqrt{y}) - 
 \sqrt{x} J_{\alpha }' (\sqrt{x}) \sqrt{y} J_{\alpha }(\sqrt{y})
 }{2(x-y)}
\\ \notag & = 
\frac{\sqrt{x} J_{\alpha +1} (\sqrt{x}) J_{\alpha } (\sqrt{y}) - 
 J_{\alpha } (\sqrt{x}) \sqrt{y} J_{\alpha +1}(\sqrt{y})
 }{2(x-y)}
\end{align}
and that for $ x=y $ 
\begin{align}& \label{:10bb}
\Ka (x,x) = \frac{1}{4}\{ J_{\alpha } (\sqrt{x}) ^2 - 
J_{\alpha +1} (\sqrt{x}) J_{\alpha -1} (\sqrt{x})  \} 
.\end{align}
Note that $ 0 \le \Ka \le \text{Id}$ 
as an operator on $ L^{2} (\SS ,dx ) $. 
By definition $ \mub $ are determinantal random point fields 
with Bessel kernels $ \Ka $ (see \cite{soshi.drpf}).

It is known that these random point fields arise as a scaling limit 
at the hard left edge of the distributions $ \mub ^{\n }$ 
of the spectrum of the Laguerre ensemble. 
The random point fields $ \mub $ represent 
the thermodynamic limit of the $ \n $-particle systems $ \mub ^{\n }$, 
whose labeled densities 
$ \mm _{\alpha }^{\n }(\mathbf{x}) d\mathbf{x}$ are given by 
\begin{align}\label{:13}& 
\mm _{\alpha }^{\n} (\mathbf{x}) =
\frac{1}{\mathcal{Z} _{\alpha }^{\n }}
 e^{-\sum_{i=1}^{\n }x_i/4\n  } \prod_{j=1}^{\n }x_j^{\alpha }
\prod_{k<l}^{\n } |x_k-x_l|^{2 } 
%\prod_{m=1}^{\n } dx_m 
.\end{align}
Very loosely, by taking $ \n $ to infinity, we obtain the following informal expression for the $ \mub $: 
\begin{align}\label{:14}&
\mub (d\mathbf{x}) = 
\frac{1}{\mathcal{Z} _{\alpha }^{\infty }}
 \prod_{j=1}^{\infty}x_j^{\alpha }
\prod_{k<l}^{\infty} |x_k-x_l|^{2 } 
\prod_{m=1}^{\infty} dx_m 
.\end{align}
Hence we regard the $ \mub $ as random point fields 
with free potentials $ \Phi _{\alpha } (x) = - \alpha \log x $ and interaction potential $ \Psi (x) = - 2\log |x|$. 
Unlike Ruelle's class of interaction potentials, one can not justify this 
using the Dobrushin-Lanford-Ruelle (DLR) equations. Instead, we will proceed in terms of logarithmic derivatives in \tref{l:23}.

We next turn to the stochastic dynamics associated with the $ \mub ^{\n }$. 
To prevent the particles from hitting the origin, we suppose that $ 1 \le \alpha $ (\lref{l:81}). 
Then, from Eq.\!  \eqref{:13}, it can be seen that the natural $ \n $-particle stochastic dynamics $ \mathbf{X}^{\n } = (X_t^{\n ,1},\ldots,X_t^{\n , \n})$ 
are given by the stochastic differential equations (SDEs) 
\begin{align}\label{:15}&
dX_t^{\n ,i} = dB_t^{i} + \{ - \frac{1}{8\n } + 
\frac{\alpha }{2X_t^{\n ,i} }  + \sum _{ j\not = i }^{\n }  
\frac{1}{X_t^{\n ,i} - X_t^{\n ,j}} \} dt 
\quad (1 \le i \le \n )
.\end{align}
Hence, taking $ \n $ to infinity, 
we come to the ISDEs 
\begin{align} \label{:16} 
& dX_t^i = dB_t^i + \{ 
\frac{\alpha }{2X_t^i }  + \sum _{ j\not = i }^{\infty}
\frac{1}{X_t^i - X_t^j} \} dt 
\quad (i \in \N )
.\end{align}
The purpose of this paper is to solve these ISDEs in such a way that the equilibrium states of the associated unlabeled dynamics 
$ \mathsf{X}_t = \sum_{i=1}^{\infty} \delta _{X_t^i}$ are Bessel random point fields $ \mub $. 

\bigskip

For a given free potential $ \Phi $ and interaction potential $ \Psi $, 
the interacting Brownian motions in infinite dimensions are the stochastic dynamics given by ISDEs of the form 
\begin{align}\label{:17a}&  
dX^i_t = dB^i_t + \frac{\beta }{2} \nabla \Phi (X_t^i) dt + 
\frac{\beta}{2}  \sum_{j\not=i} 
\nabla \Psi (X^i_t,X^j_t) dt \quad (i\in\N )
.\end{align}
Here, $ \{ B^i \}_{i\in\N } $ is a sequence of independent copies of $ d$-dimensional Brownian motions. 
% and $ \mathbf{X} = (X^i)_{i\in\N }$ is a continuous $ (\Rd )^{\N }$-valued process. 
The study of interacting Brownian motions in infinite dimensions 
 was initiated by Lang \cite{lang.1}, \cite{lang.2}, 
and continued by Shiga \cite{shiga}, Fritz \cite{Fr}, Tanemura \cite{T2}, and others. 
In these works, $ \Psi $ is assumed to be a Ruelle type potential: that is, 
$ \Psi $ is super-stable and integrable at infinity. 
In addition, $ \Psi $ is assumed to be of class $ C^3_0$ 
(\cite{lang.1,lang.2,shiga,Fr}) or to decay exponentially at infinity with a hard core (\cite{T2,T-R}). Hence, polynomial decay potentials are excluded,
 even from Ruelle's category. 

Recently, an interesting class of random point fields has appeared from random matrix theory. 
This class includes such as the sine, Airy, and Bessel random point fields in 
one-dimensional space and the Ginibre random point field in two dimensions. 
These represent the thermodynamic limits of the distributions of Gaussian random matrices. There are many other such random point fields that emerge from random matrix theory, but these examples are of particular note. The sine, Airy, and Bessel random point fields describe the universality classes called bulk, soft-edge, and hard-edge scaling limits, respectively. 
The Ginibre random point field is rotation and translation invariant, and thus is the typical example in two dimensions.

In these random point fields, the interactions always have logarithmic  potentials and therefore represent the outer side of the classical theory of interacting Brownian motions in infinite dimensions. 
In \cite{o.tp, o.isde, o.rm, o.rm2}, 
the second author (H.O.) developed the theory applicable to these examples. 
This theory asserts that the quasi-Gibbs property and the existence of 
logarithmic derivative $ \dmu $ of random poitnt fields $ \mu $ 
together with marginal assumptions such as non-collision 
and non-explosion properties of tagged particles imply the existence of (weak) solutions 
of the ISDEs in Eq. \eqref{:26r}. 
In \cite{o.isde, o.rm, o.rm2}, he also gave a sufficient condition of the quasi-Gibbs property and 
the existence of a logarithmic derivative. 
Nevertheless, veryfying this sufficient condition for 
random point fields appearing in 
random matrix theory is a difficult problem, 
and the proof depends crucially on the specific property of each model. 

In \cite{o.isde, o.rm}, H.O.\! proved these properties for 
the sine and Ginibre random point fields and solved ISDEs 
related to these random point fields in the sense of weak solutions. 
In \cite{o-t.airy}, H.O.\! and Tanemura prove the quasi-Gibbs property and 
calculate the logarithmic derivative of the Airy random point fields. 
In \cite{o-t.tail}, they develop a general theory on ISDEs that asserts the existence and 
pathwise uniqueness of  strong solutions of ISDEs under the assumptions of 
the quasi-Gibbs property, the existence of logarithmic derivative, 
and other marginal assumptions. 

The most important assumptions for the theories 
 in both \cite{o.tp,o.isde,o.rm,o.rm2} and \cite{o-t.tail} 
 are the same. 
These are the quasi-Gibbs property and the existence of a 
logarithmic derivative $ \dmu $ of the random point fields $ \mu $. 
Once these have been established, 
we can solve the ISDEs of Eq. \eqref{:26r} in the sense of weak solutions 
using \cite{o.tp, o.isde, o.rm, o.rm2} 
and in the sense of pathwise unique, strong solutions by \cite{o-t.tail}. 
Hence, ensuring that these two assumptions holds is an important issue. 

In the present paper, we prove the quasi-Gibbs property and 
calculate the logarithmic derivative for the Bessel random point field. 
Applying the general theories in \cite{o.tp, o.isde, o.rm, o.rm2, o-t.tail}, 
we then solve the ISDE for the Bessel random point field with $ \beta = 2 $, 
which describes the remaining universality class in one dimension.

\bigskip

This paper is organized as follows: In \sref{s:2}, we establish the mathematical framework and state the main results (Theorems \ref{l:21}--\ref{l:24}). 
In \sref{s:3}, we prove Theorems \ref{l:21} and \ref{l:22} using Theorems \ref{l:23} and \ref{l:24} in combination with the general theory developed in \cite{o.tp,o.isde,o.rm,o.rm2,o-t.tail}. 
In \sref{s:4}, we set forth \tref{l:41} in preparation for \sref{s:5}, 
where we calculate the logarithmic derivatives of Bessel random point fields and prove \tref{l:23}. 
In \sref{s:6}, we prove that these are quasi-Gibbssian (\tref{l:24}). 
In \sref{s:7}, we prove \lref{l:50}. 
In \sref{s:8}, we prove \lref{l:81}. 

\section{ Set up and main results}\label{s:2}%%[][]
Let $ \SS = \ZI  $ and $ \Sr = \{ x\in \SS;\,  x < r \} $. 
Let 
$$ \SSS  = \{ \mathsf{s} = \sum _i \delta _{s_i}\, ;\, 
s_i \in \SS ,\, 
\mathsf{s} ( \Sr ) < \infty \text{ for all } r\in\N  \} 
,$$
where $ \delta _{a} $ stands for the delta measure at $ a $. 
We endow $ \SSS   $ with the vague topology, under which 
$ \SSS   $ is a Polish space. 
$ \SSS $ is called the configuration space over $ \SS $. 
We write $ \sss (x)=\sss (\{ x \} )$. Let 
\begin{align}\label{:20a}&
\SSSsi =
\{ \sss \in \SSS \, ;\, \, \sss (x)\le 1 \text{ for all }x \in \SS  ,\, \, 
\sss (\SS )= \infty \} 
.\end{align}
By definition, $ \SSSsi $ is the set of the configurations consisting of an infinite number of single point measures.

A symmetric locally integrable function 
$ \map{\rho ^n }{\SS ^n}{\ZI  } $ is called 
the $ n $-point correlation function of a probability measure $ \mu $ 
on $ \SSS $ w.r.t.\ the Lebesgue measure if $ \rho ^n $ satisfies 
\begin{align}\label{:20c}&
\int_{A_1^{k_1}\ts \cdots \ts A_m^{k_m}} 
\rho ^n (x_1,\ldots,x_n) dx_1\cdots dx_n 
 = \int _{\SSS } \prod _{i = 1}^{m} 
\frac{\mathsf{s} (A_i) ! }
{(\mathsf{s} (A_i) - k_i )!} d\mu
 \end{align}
for any sequence of disjoint bounded measurable subsets 
$ A_1,\ldots,A_m \subset \SS $ and a sequence of natural numbers 
$ k_1,\ldots,k_m $ satisfying $ k_1+\cdots + k_m = n $. 
When $ \mathsf{s} (A_i) - k_i  < 0$, according to our interpretation, 
${\mathsf{s} (A_i) ! }/{(\mathsf{s} (A_i) - k_i )!} = 0$ by convention. 
It is known that under a mild condition 
$ \{ \rho ^n \}_{n \in \N }$ determines the measure $ \mu $ \cite{soshi.drpf}.

Let $ \mub $ be Bessel random point fields. 
By definition $ \mub $ are probability measures on $ \SSS $ 
whose $ n $-point correlation functions $ \rbn $ are given by \eqref{:10a}.

Let $ \map{\ulab }{\SS ^{\mathbb{N}}  }{\SSS }$ such that 
$ \ulab ((s_i)) = \sum_{i=1}^{\infty} \delta _{s_i}$. We call $ \ulab $ an unlabel map. 

A (weak) solution $ (\mathbf{X},\mathbf{B})$ of an ISDE starting at $ \mathbf{s}$ 
is called a strong solution 
if $ \mathbf{X}$ is a function of Brownian motion $ \mathbf{B}$ and the starting point 
$ \mathbf{s}$. 

For a pair of Radon measures $ \mu $ and $ \nu $, we write $ \mu \prec \nu $ 
if $ \mu $ is absolutely continuous with respect to $ \nu $. 
\begin{thm}\label{l:21} %% Dyson 
Assume that $ 1 \le \alpha < \infty $. The following then holds. \\
\thetag{1} 
For each $ \alpha $, there exists a set $ \SSSdys $ such that 
\begin{align}\label{:21a}&
\mub (\SSSdys )= 1, \quad \SSSdys \subset \SSSsi 
,\end{align}
and that, for all $  \mathbf{s}\in \ulab ^{-1}(\SSSdys )$, 
there exists a $ \ZI  ^{\N }$-valued continuous process 
$ \mathbf{X}=(X^i)_{i\in\N }$, and $ \R ^{\N }$-valued Brownian motion 
$ \mathbf{B}=(B^i)_{i\in\N }$ satisfying 
\begin{align} \label{:21b} 
& dX_t^i = dB_t^i + \{ 
\frac{\alpha }{2X_t^i }  + \sum _{ j\not = i }^{\infty} 
\frac{1}{X_t^i - X_t^j} \} dt 
\quad (i \in \N )
,\\\label{:21c}&
\mathbf{X}_0 = \mathbf{s} 
.\end{align}
Moreover, $ \mathbf{X}$ satisfies 
\begin{align}\label{:21f}&
P (\ulab (\mathbf{X}_t) \in \SSSdys ,\ 0\le \forall t < \infty ) = 1 
.\end{align}
\thetag{2} 
For $ \mub \circ \lab ^{-1}$-a.s.\!\! $ \mathbf{s}$, 
$ (\mathbf{X},\mathbf{B})$ above 
is a strong solution of \eqref{:21b} and \eqref{:21c} such that 
\begin{align}\label{:21g}& 
\mub \circ \ulab (\mathbf{X}_t) ^{-1} \prec \mub \quad \text{ for all } t \in [0,\infty)
.\end{align}
Furthermore, 
the $ \mub $-strong uniqueness holds in the sense that 
any family of weak solutions satisfying \eqref{:21g} becomes the strong solution 
for $ \mub \circ \lab ^{-1}$-a.s.\!\! $ \mathbf{s}$, and that 
any two strong solutions 
$ (\mathbf{X}, \mathbf{B})$ and $ (\mathbf{X}',\mathbf{B})$ defined on the same Brownian motion $ \mathbf{B}$ starting at $ \mathbf{s} = \lab (\mathsf{s}) $ 
satisfying \eqref{:21g} are pathwise unique 
\begin{align}\label{:21h}&
P(\mathbf{X}_t=\mathbf{X}_t' \text{ for all }t \ ) = 1 
\end{align}
for $ \mub \circ \lab ^{-1}$-a.s.\!\! $ \mathbf{s}$. 
Here 
$ \map{\lab }{\SSS }{[0,\infty)^{\N }}$ is the label introduced 
by \rref{r:21} \thetag{2}. 
\end{thm}

\begin{rem}\label{r:21} \thetag{1}
When $ -1 < \alpha < 1 $, the left most particle hits the origin. Hence a coefficient coming from the boundary condition will appear in the ISDEs. 
Since we suppose $ 1 \le \alpha $, particles never hit the origin (see \lref{l:81}). 
It would be an interesting problem to study the case $ -1 < \alpha < 1$ 
where the boundary condition would appear. 
\\
\thetag{2}
The correspondence $ \mathsf{s}\mapsto \mathbf{s}=(s_i)$ is called a label. 
In case of Bessel random point fields, there exists a natural label such that $ s_{i} < s_{i+1}$ for all $ i \in \mathbb{N}$. 
As we see later, all the particles $ X_t^i$ never collide each other for all $ t \in [0,\infty)$. Hence the initial label is kept forever. 
In particular, if we take $ X_0^i< X_0^{i+1}$ for all $ i \in \mathbb{N}$ initially, then $ X_t^i< X_t^{i+1}$ for all $ i \in \mathbb{N}$ and 
$ t \in [0,\infty)$. We denote this label by $ \lab $ in \thetag{2} of \tref{l:21}. 
\\\thetag{3} 
After submitting the first version of the manuscript, a general theory 
on the existence and uniqueness of 
strong solutions of infinite-dimensional stochastic differential equations 
has been developed and completed (see \cite{o-t.tail}). 
%In the version \cite{o-t.tail} of the general theory, 
%In \cite{o-t.tail}, 
When the first version was submitted, 
no preprint of the general theory was available. 
It is now clarified and confirmed that 
the main assumptions required in \cite{o-t.tail} 
follow from the results of the present paper 
(Theorems \ref{l:23} and \ref{l:24}). 
We noticed that the existence and uniqueness of strong solutions 
is obtained immediately by combining the general theory 
in \cite{o-t.tail} and the results in Theorems \ref{l:23} and \ref{l:24}. 
Hence, in the revised version, we newly add \thetag{2} of \tref{l:21}. 
%by applying the above-mentioned general theory in \cite{o-t.tail}. 
The original \tref{l:21} is renumbered as \tref{l:21} \thetag{1}.  
\end{rem}

A diffusion with state space $ S_0$ is a family of continuous stochastic processes with the strong Markov property starting at each point of the state space $ S_0$. 
In general, the notion of the Markov property depends on the filtering. 
We always consider the natural filtering in the present paper \cite{fot2}. 

\begin{thm} \label{l:22} 
Assume that $ 1 \le \alpha < \infty $. 
Let $ \SSSSdys = \ulab ^{-1}(\SSSdys )$. 
Let $ \mathbf{P}_{\mathbf{s}}$ be the distribution 
of $ \mathbf{X}$ given by \tref{l:21}. Then 
$\{\mathbf{P}_{\mathbf{s}} \}_{\mathbf{s}\in\SSSSdys }$ 
is a diffusion with state space $ \SSSSdys $.
\end{thm}

%\bigskip 

We deduce \tref{l:21} and \tref{l:22} 
from a general theory developed in \cite{o.tp,o.isde,o.rm,o.rm2,o-t.tail}. 
The key point for this is to calculate {\em the logarithmic derivative} 
of the measure $ \mub $ 
and to prove {\em the quasi-Gibbs property} of $ \mub $. 
These two notions play an important role in the proof of 
\tref{l:21} and \tref{l:22}. 

The logarithmic derivative of $ \mub $ will be calculated in \tref{l:21}. 
We will use \tref{l:41} to prove \tref{l:21} in \sref{s:5}. 
The quasi-Gibbs property of $ \mub $ will be proved in \tref{l:24}.  
\tref{l:24} will be proved in \sref{s:6}.

To introduce the notion of the logarithmic derivative of random point fields 
we recall the definitions of reduced Palm measures and Campbell measures. 

Let $ \mu $ be a probability measure on $ (\SSS , \mathcal{B}(\SSS ))$.  
A probability measure $ \mu _{\mathbf{x}} $ is called 
the reduced Palm measure conditioned at 
$ \mathbf{x} =(x_1,\ldots,x_k)  \in \Sk $ if $ \mu _{\mathbf{x}} $ is 
the regular conditional probability defined by 
\begin{align}\label{:22x}&
\mu _{\mathbf{x}} = \mu (\cdot - \sum_{i=1}^{k} \delta _{x_i} | \ 
\sss ( x_i )\ge 1 \text{ for }i=1,\ldots,k)
.\end{align}
Let $ \rho ^{k}$ be the $ k $-point correlation function of 
$ \mu $ with respect to the Lebesgue measure. 
Let $ \muk $ be the measure on $ \SkS $ defined by 
\begin{align}\label{:22y}& 
\muk (A\ts B )= \int_{A}\mu _{\mathbf{x}} (B) \rho ^{k}( \mathbf{x} ) d \mathbf{x} 
.\end{align}
Here we set $ d \mathbf{x} = dx_1\cdots dx_k $ for 
$  \mathbf{x} =(x_1,\ldots,x_k) \in \Sk $. 
The measure $ \muk $ is called the $ k$-Campbell measure.  

\begin{dfn}\label{dfn:1}
We call $ \dmu \in \Llocone (\muone ) $ 
the logarithmic derivative of $ \mu$ if $ \dmu $ satisfies 
\begin{align}\label{:22z}&
\int_{\SoneSSS } \dmu f d \muone  = - \int_{\SoneSSS }
\PD{f(x,\sss )}{x} d \muone 
\quad \text{ for all } f \in C^{\infty}_{0} ((0,\infty))\ot C_b(\SSS ) 
.\end{align}
\end{dfn}
Very loosely, \eqref{:22z} can be written as 
$  \dmu  = {\partial \log \muone (x,\sss )}/{\partial x} $. 
This intuitive expression is the reason why we call $ \dmu $ the logarithmic derivative of $ \mu $. 

\begin{thm} \label{l:23} 
Assume that $ 1 \le \alpha < \infty $. 
Then $ \mub $ has the logarithmic derivative 
$ \dlog ^{\mub }  \in \Lloctwo (\mub ^{1}) $ defined by 
\begin{align}\label{:b23a}&
\dlog ^{\mub }(x,\mathsf{y})= \frac{\alpha }{x} + 
\sum_{i\in\N }\frac{2}{x-y_i}
.\end{align}
Here $\mathsf{y}=\sum_{i\in\N }\delta _{y_i}$. 
\end{thm}

\begin{rem}\label{r:23} 
Since we suppose $ 1 \le \alpha $, particles never hit the origin in \tref{l:23}. 
If $ -1 < \alpha < 1 $, then the left most particle hits the origin. 
Hence in the definition of logarithmic derivative, it would be more natural 
to take $ C^{\infty}_{0} ([0,\infty))\ot C_b(\SSS ) $ as a space of test functions. 
In this case, the logarithmic derivative contains a term arising from the bounadry condition. 
Although it would be interesting to study this case, we do not pursue this here.
\end{rem}

We next introduce the notion of the quasi-Gibbs property. 
 
For two measures $ \nu _1,\nu _2 $ on a measurable space 
$ (\Omega , \mathcal{B})$ we write $ \nu _1 \le \nu _2 $ 
if $ \nu _1(A)\le \nu _2(A)$ for all $ A\in\mathcal{B}$. 
We say a sequence of finite Radon measures $ \{ \nu ^{\n }  \} $ 
on a Polish space $ \Omega $ converge weakly to 
a finite Radon measure $ \nu $ 
if $ \lim_{ \n \to \infty} \int fd\nu ^{\n }  = \int f d\nu $ 
for all $ f \in C_b(\Omega ) $.

Let $ \{ b_r \}_{r=1}^{\infty} $ be an increasing sequence of natural numbers. 
Let 
\begin{align*}&
\text{$ \SSS _r^m = \{ \mathsf{x} \in \SSS \, ;\, 
 \mathsf{x} (\SS _{b_r} ) = m \} $ and 
 $ \Lambda _r^m = \Lambda (\cdot \cap \SSS _r^m )$}
,\end{align*}
where $ \Lambda $ is the Poisson random point field 
 whose intensity is the Lebesgue measure. 
We denote by $ \mathcal{H}_{ r }(\mathsf{x} ) $ the Hamiltonian on $ \SS _{b_r}$ such that 
\begin{align}\label{:qg5}&
\mathcal{H}_{ r }(\mathsf{x} ) = 
\sum_{x_i\in \SS _{b_r}} \Phi (x_i) + 
\frac{1}{2} \sum_{x_i,x_j \in \SS _{b_r} ,\, i\not= j} 
 \Psi (x_i,x_j)
.\end{align}
Here we set $ \mathsf{x} = \sum_i \delta _{x_i}$. 
Let $ \map{\pi _r, \pi _r^c}{\SSS }{\SSS }$ be the maps such that 
\begin{align}& \notag 
\text{$ \pi _r (\mathsf{s})(\cdot ) = \mathsf{s}(\cdot \cap \SS _{b_r} )$ and 
 $ \pi _r^c (\mathsf{s})(\cdot ) = \mathsf{s}(\cdot \cap \SS _{b_r}^c )$. }
\end{align}
\begin{dfn}\label{dfn:2}
A probability measure $ \mu $ is said to be 
a $ (\Phi , \Psi ) $-quasi Gibbs measure if 
there exists an increasing sequence of natural numbers $ \{ b_r \}_{r=1}^{\infty} $ 
such that, for each $ r,m \in \N $ and for $ \mu _{r,k}^m $-a.e.\,  $\sss \in \SSS $, 
there exists a sequence of Borel subsets $ \SSS _{r,k}^m $ satisfying 
\begin{align} & \notag 
\SSS _{r,k}^m \subset \SSS _{r,k+1}^m\subset \SSS _{r}^m \quad \text{ for all }k ,
\\ \notag &
\limi{k} \mu (\cdot \cap \SSS _{r,k}^m ) = \mu (\cdot \cap \SSS _{r}^m ) 
\quad \text{ weakly}
,\end{align}
and that $  \mu _{r,k}^m = \mu (\cdot \cap \SSS _{r,k}^m )$ satisfy, 
for each $ r,m , k \in \N $ and for $ \mu _{r,k}^m $-a.e.\,  $\sss \in \SSS $, 
\begin{align}\label{:qg2}&
\cref{;2y}^{-1} 
e^{- \mathcal{H}_{ r }(\mathsf{x} )  } \Lambda _r^m (d\mathsf{x})
 \le 
\mu _{r,k,\mathsf{s}}^{m}(d\mathsf{x}) 
\le 
\cref{;2y} 
 e^{- \mathcal{H}_{ r }(\mathsf{x} )}
\Lambda _r^m (d\mathsf{x})
.\end{align}
Here $\Ct \label{;2y} = \cref{;2y}(r,m,k,\pi _{ r }^c (\mathsf{s}))$ 
is a positive constant and $ \mu _{r,k,\mathsf{s}}^{m}$ 
is the regular conditional probability measure of $ \mu _{r,k}^m$ defined by 
\begin{align} \label{:qg4} 
& \mu _{r,k,\mathsf{s}}^{m}(d\mathsf{x}) = \mu _{r,k}^m(\pi _{ r } 
\in d\mathsf{x} | \ \pi _{ r }^c ) (\mathsf{s}) 
.\end{align}
\end{dfn}
The notion of quasi-Gibbsian is first introduced in \cite{o.rm}. 
The original definition of the quasi-Gibbs measures 
 is slightly more general than the present version, and is essentially the same. 
We adopt here a restrictive version for the sake of simplicity. 

We remark that we do not assume the symmetry of the interaction potential $ \Psi $. Hence we take the ordered summation of $ \Psi (x_i,x_j)$, 
and put $ 1/2$ in the sum of \eqref{:qg5}. 

\begin{thm} \label{l:24} 
Let $ -1 < \alpha < \infty $. 
Then $ \mub $ is a $ (\alpha \log x , 2 \log |x-y| ) $-quasi Gibbs measure. 
\end{thm}

Combining \tref{l:24} with a general theory \cite[Corollary 2.1]{o.rm}, 
we obtain a natural unlabeled $ \mub $-reversible diffusion 
$ (\mathsf{X},\mathsf{P})$.  
\begin{thm} \label{l:25} 
Let $ -1 < \alpha < \infty $. 
Let $ \mathcal{E}^{\mub } $ and 
$ \di ^{\mub }$ be as in \eqref{:26j} and \eqref{:26J} 
 with $ k=0 $ and $ \mu = \mub $, 
Then $ (\mathcal{E}^{\mub },\di ^{\mub })$ is closable on $ L^2 (\SSS , \mub )$. There exists a diffusion $ (\mathsf{X},\mathsf{P})$ 
associated with the closure of 
$ (\mathcal{E}^{\mub },\di ^{\mub })$ on $ L^2 (\SSS , \mub )$. 
\end{thm}

\begin{rem}\label{r:25} 
\thetag{1} If $ -1 < \alpha < 1$, then the left most particle hits the origin. 
\\
\thetag{2} 
We write the diffusion $ \mathsf{X}$ in \tref{l:25} as 
$ \mathsf{X}_t = \sum_{i\in\N } \delta_{X_t^i}$. 
Since the particles never collide each other, 
the infinite-dimensional labeled paths $ (X_t^i)_{i\in\N }$ 
is well defined. 
Then with suitable labeling of the unlabeled particles $ \mathsf{X}$ 
at time $ t=0$, 
the solution of the ISDE $ (X_t^i-X_0^i)_{i\in\N }$ 
becomes an infinite-dimensional additive functional 
of the unlabeled diffusion $ (\mathsf{X},\mathsf{P})$. 
We remark that this additive functional is {\em not} Dirichlet process 
because {\em no} coordinate functions $ x_i $ ($i\in\N $) 
 belong to the domain of the Dirichlet form even if locally. 
 \\
\thetag{3} 
There are other approaches for this kind of 
unlabeled stochastic dynamics related to random matrix theory. 
See \cite{sp.2}, \cite{kt.cmp}, \cite{kt.11}, \cite{boro-ols.12}, \cite{boro-gorin.13}, and \cite{ols.11}. 
These approaches are more algebraic, and restricted to one dimensional system with inverse temperature $ \beta = 2 $. 
\\\thetag{4} 
Let $ \kpath $ be the map from $ C([0,\infty);\SS ^{\mathbb{N}})$
 to $ C([0,\infty);\SSS )$ defined by 
\begin{align}\label{:20b}&
\kpath (\mathbf{X}) = \{ \sum_{i=1}^{\infty} 
\delta _{X_t^i}\}_{t\in[0,\infty)} = 
\{\ulab (\mathbf{X}_t) \}_{t\in[0,\infty)}
,\end{align}
where $ \mathbf{X} = \{(X^i_t)_{i=1}^k \} $. We set 
$ \mathsf{X} = \kpath (\mathbf{X})$. 
We call $ \mathbf{X} $ (resp. $ \mathsf{X} $) the labeled process (unlabeled process). 
Then the relation between the labeled process $ \mathbf{X}$ in \tref{l:21} 
and the unlabeled process $ \mathsf{X}$ in \tref{l:25} is 
that $ \kpath (\mathbf{X}) = \mathsf{X}$ for a suitable version of 
the processes with quasi-everywhere starting points.  
This identity is a corollary of \cite[Theorem 2.4]{o.tp}. 
\end{rem}

\section{Proof of Theorems \ref{l:21} and \ref{l:22}}\label{s:3}
The purpose of this section is to prove Theorems \ref{l:21} and \ref{l:22}. 
For this we will use \tref{l:23} and \tref{l:24}, and 
a result from \cite{o.isde} in a reduced form being 
sufficient for the present problem. 

Let $ \mathbf{X}=(X_t^i)_{i\in\mathbb{N}}$ be a labeled process as before. 
For $ k \in \{ 0 \}\cup \N  $, 
the process $ (X_t^1,\ldots,X_t^k, \sum_{j>k}^{\infty}\delta_{X_t^j})$ 
is said to be a $ k$-labeled process. 
When $ k=0$, the $ k$-labeled process equals the unlabeled process $ \kpath (\mathbf{X}) $. 

We introduce Dirichlet forms describing the $ k $-labeled process. 
For a subset $ A \subset \SS $ we define the map 
$ \map{\pi _{A }}{\SSS }{\SSS } $ by 
$ \pi _{A } (\sss  ) = \sss ( A \cap \cdot )  $. 
We say a function $ \map{f}{\SSS }{\R } $ is local if $ f $ is 
$ \sigma[\pi _{ A }]$-measurable for some compact set $ A \subset \SS $. 
We say $ f $ is smooth if $ \tilde{f} $ is smooth, 
where $ \tilde{f}((s_i)) $ is the permutation invariant function in $ (s_i) $ such that 
$ f (\sss  ) = \tilde{f} ((s_i)) $ for $ \sss  = \sum _i \delta _{s_i} $. 

Let $ \di $  be the set of all local, smooth functions on $ \SSS $. 
For $ f,g \in \di $ we set $ \map{\DDD [f,g]}{\SSS }{\R } $ by 
\begin{align} \label{:30a} & 
\DDD [f,g](\sss ) = 
\frac{1}{2} \sum _{ i } 
\PD{\widetilde{f}(\mathbf{s})}{s_{i}} \PD{\widetilde{g}(\mathbf{s})}{s_{i}}
.\end{align}
Here $ \sss = \sum_{i}\delta _{s_i}$ and $ \mathbf{s}=(s_i)$. 
For given $ f $ and $ g $ in $ \di $, it is easy to see that the right-hand side of \eqref{:30a} depends only on $ \sss $. So $ \DDD [f,g]$ is well defined. %
For $ f,g \in C_0^{\infty}(\Sk )\ot \di $ let $\nabla ^{k}[f,g]$ be the function on $ \SkS $ defined by 
\begin{align}\label{:26h}&
\nabla ^{k}
[f,g ] (\mathbf{x},\sss ) = 
\frac{1}{2} \sum _{j=1}^{k}
\PD{f(\mathbf{x}, \sss )}{x_{j}}\PD{g(\mathbf{x}, \sss )}{x_{j}}
.\end{align}
where $\mathbf{x} =(x_j)\in \Sk $. We set $ \DDDk  $ for $ k\ge 1$ by 
\begin{align}\label{:26i}&   
\DDDk [f,g](\mathbf{x},\sss ) = 
\nabla ^{k}[f,g] (\mathbf{x},\sss )
+
\DDD [f(\mathbf{x},\cdot ),g(\mathbf{x},\cdot )](\sss )
.\end{align}
Let $ (\Eak ,\dik )$ be the bilinear form defined by  
\begin{align}\label{:26j} &
\Eak (f,g) = \int_{\SkS }  \DDDk [f,g] d\muk 
,\\\label{:26J}& 
\dik = \{ f \in C_0^{\infty}(\Sk )\ot \di \cap  L^2(\SS ^k \ts \SSS , \muk ) 
\, ;\, \, 
\Eak (f,f) < \infty 
\} 
.\end{align}
When $ k=0 $, we take $ \DDDzero =\DDD $, 
$ \mu ^{0} = \mu $, and $ \Eazero = \Ea $. 
We set $ \Lm = L^2(\SSS , \mu)$ and $ \Lmuk = L^{2}(\SkS ,\muk ) $ and so on.

%\medskip 

We assume that there exists a probability measure $ \mu $ on $ \SSS $ 
with correlation functions $ \{ \rho ^{k} \}_{k\in\N } $ satisfying 
\Ass{A.1}--\Ass{A.5}: 
\\
\Ass{A.1} $ \rho ^{k}$ is locally bounded for each $ k \in \N $. 
\\
\Ass{A.2} There exists a logarithmic derivative $ \dmu $ in the sense of \eqref{:22z}. 
\\%
\Ass{A.3} 
$(\Eak ,\dik )$ is closable on $ \Lmuk $ for each $ k\in\{ 0 \}\cup \N $. 
\\
\Ass{A.4} 
$\mathrm{Cap}^{\mu } (\{\SSSsi \}^c) = 0 $.
\\
\Ass{A.5} 
There exists a $ T>0 $ such that for each $ R>0 $ 
\begin{align}\label{:26p}&
\liminf_{r\to \infty}\ ( \{ \int_{|x|\le {r+R}} \rho ^1 (x)dx \}
\{ \int_{\frac{r}{\sqrt{(r+R )T}}}^{\infty} e^{-u^{2}/2}du \} ) = 0
.\end{align}

\medskip

Let $ (\Eak ,\dak )$ be the closure of $ (\Eak ,\dik )$ on $ \Lmuk $. 
It is known \cite[Lemma 2.3]{o.tp} that $ (\Eak ,\dak )$ is quasi-regular and that the associated diffusion $ (\PPk , \mathsf{X}^k )$ exists. We refer to \cite{mr} for the definition and 
necessary background of quasi-regular Dirichlet forms.  
We remark that $ \mathrm{Cap}^{\mu }$ in \Ass{A.4} is  
 the capacity of the Dirichlet space $(\Eazero ,\d ^{\mu } ,\Lm )$. 

The assumptions \Ass{A.4} and \Ass{A.5} have clear dynamical interpretations. Indeed, \Ass{A.4} means that particles never collide with each other. 
Moreover, \Ass{A.5} means that no labeled particle ever explodes \cite{o.tp}. 
%See \cite{fot} for the necessary background of the Dirichlet form theory for these facts. 

We quote two theorems from \cite{o.isde}. 
\begin{thm}[{\cite[Theorem 26]{o.isde}}]\label{l:31} %% 
Assume \Ass{A.1}--\Ass{A.5}. Then there exists an 
$ \SSS _0 $ such that 
\begin{align}\label{:26q}&
\mu (\SSS _0 )= 1, \quad \SSS _0 \subset 
\SSSsi 
,\end{align}
and that, for all $  \mathbf{s}\in \ulab ^{-1}(\SSS _0 )$, there exists 
an $\SS ^{\N }$-valued continuous process $ \mathbf{X}=(X^i)_{i\in\N }$, and 
$(\R )^{\N }$-valued Brownian motion $ \mathbf{B}=(B^i)_{i\in\N }$ satisfying 
\begin{align}\label{:26r}&
dX^i_t = dB^i_t +\frac{1}{2}\dmu (X^i_t,\mathsf{X}^{i*}_t)dt \quad (i\in \N )
,\\ \label{:26s}&
\mathbf{X}_0 = \mathbf{s}
.\end{align}
Moreover, $ \mathbf{X}$ satisfies %
\begin{align}\label{:26t}&
P (\ulab (\mathbf{X}_t) \in \SSS_0 ,\ 0\le \forall t < \infty ) = 1 
.\end{align}
\end{thm}

\begin{thm}[{\cite[Theorem 27]{o.isde}}] \label{l:32}
Let $ \SSSS _0 $ be the subset of $ \SS ^{\N }$ defined by 
$ \SSSS _0 =\ulab ^{-1}(\SSS _0 )$. Let $ \mathbf{P}_{\mathbf{s}}$ be the distribution of $ \mathbf{X}$ given by \tref{l:31}. Then 
$\{\mathbf{P}_{\mathbf{s}}\}_{\mathbf{s}\in\SSSS _0 }$ 
is a diffusion with state space $ \SSSS _0 $.
\end{thm}

We take $ \mu = \mub $. 
Then the assumptions \Ass{A.1}, \Ass{A.4}, and \Ass{A.5} are easily checked as we see in the next lemma. 
\begin{lem} \label{l:33}
$ \mub $ satisfy \Ass{A.1}, \Ass{A.4}, and \Ass{A.5}.
\end{lem}
\aaaaa 
 \Ass{A.1} and \Ass{A.5} are clear because 
 the correlation functions $\{\rbn \}$ of $\mub $ 
 are given by the equation \eqref{:10a} 
 and the kernels $ \Kan $ are locally bounded in $ (0,\infty)$ 
 and bounded in $ [1,\infty)$. 
 \Ass{A.4} follows from \cite[Theorem 2.1]{o.col} 
 because the kernel $ \Kan $ is locally Lipschitz continuous. 
\bbbbb 

We next deduce 
Theorems \ref{l:21} and \ref{l:22} from 
Theorems \ref{l:23} and \ref{l:24}.

\bigskip 

\noindent {\em Proof of \tref{l:21} \thetag{1} and \tref{l:22}. } 
We will use Theorems \ref{l:31} and \tref{l:32} to prove 
\tref{l:21} \thetag{1} and \tref{l:22}. 
For this we check the assumptions \Ass{A.1}--\Ass{A.5} 
with a help of Theorems \ref{l:23} and \ref{l:24}. 

The assumption \Ass{A.2} follows from \tref{l:23}. %
From \lref{l:33} we have already known that 
$ \mub $ satisfy \Ass{A.1}, \Ass{A.4}, and \Ass{A.5}. 

From \tref{l:24} we see that $ \mub $ are quasi-Gibbssian with continuous potentials. 
In \cite[Lermma 3.6]{o.rm}, it was proved that, 
when potentials are upper semi-continuous, 
the closability in \Ass{A.3} for $ k=0 $ follows 
from the quasi-Gibbs property. 
Then we have \Ass{A.3} for $ k=0 $. 
The closability for general $ k \ge 1 $ also follows from the quasi-Gibbs property of $ \mub $ in a similar fashion. 
Hence we obtain \Ass{A.3} for $ \mub $. 

We have thus seen that the assumptions \Ass{A.1}--\Ass{A.5} 
are fulfilled. Hence, \tref{l:21} \thetag{1} and \tref{l:22} follows from 
Theorems \ref{l:31} and \ref{l:32}, respectively. 
\qed

\bigskip 

We next prove \thetag{2} of \tref{l:21} using a result in \cite{o-t.tail}. 

\noindent 
\noindent {\em Proof of \tref{l:21} \thetag{2}. } 
We deduce \tref{l:21} \thetag{2} from Theorem 9.3 in \cite{o-t.tail}. 
We check the assumptions in \cite[Theorem 9.3]{o-t.tail}. 
These are labeled in \cite{o-t.tail} as follows: 
\As{A1}--\As{A4}, \As{A5'}, \As{A8'}, \As{A9}, 
\As{E1}, \As{F1}, and \As{F2}.

We see that \As{A3}, \As{A1}, \As{A4}, and \As{A5'} in \cite{o-t.tail} follow from 
\As{A.1}, \As{A.2}, \As{A.4}, and \As{A.5} in the present paper, respectively.  
We deduce \As{A2} in \cite{o-t.tail} from \tref{l:24} immediately. 
\As{A8'} follows from \eqref{:10bb}.  
Assumption \As{A9} in \cite{o-t.tail} asserts that $ \mub $ is tail trivial. 
This was proved in \cite{o-o.tt}. 
%
% which is a generalization of the result for discrete spaces 
%proved by Shirari-Takahashi \cite{ST2} and Russel Lyons \cite{rl.dpm}. 
%\As{A9} thus follows from \cite{o-o.tt}. 

For $ \mathsf{k},r \in \N $, let $ a_{\mathsf{k}}(r) = \mathsf{k} \sqrt{r}$. 
Set $ a_{\mathsf{k}} = \{ a_{\mathsf{k}} (r)\}_{r\in\N } $ and 
$ \mathbf{a} = \{ a_{\mathsf{k}} \}_{\mathsf{k}\in\N } $. 
Let $ \mathsf{K}[\mathbf{a}] = \cup_{r=1}^{\infty} 
\mathsf{K}[a_{\mathsf{k}}]$, where  
$ \mathsf{K}[a_{\mathsf{k}}]= 
\{ \mathsf{s}; \mathsf{s}(\Sr ) \le a_{\mathsf{k}}(r) \text{ for all } r \in \N  \} $. 
We then deduce from \eqref{:10bb} that 
\begin{align}& \notag %\label{:34}& 
\mub (\mathsf{K}[\mathbf{a}]) = 1 
.\end{align}
This corresponds to \As{E1} in \cite{o-t.tail}. 
Assumptions \As{F1} and \As{F2} depend on 
a nonnegative integer $ \ell $ in \cite[Theorem 9.3]{o-t.tail}. 
We take $ \ell = 1 $ here. 
We can easily check \As{F1} and \As{F2} by a straightforward calculation. 
Indeed, $ \sigma $ and $ b $ in \As{F1} and \As{F2} 
in \cite{o-t.tail} become 
$ \sigma = 1 $ and $ b = 2\dlog ^{\mub } $. Hence, $ \partial_x \sigma = 0 $ and, 
from \tref{l:23},  
\begin{align*}&
\partial_x b (x,\mathsf{s}) = -\frac{2\alpha }{x^2} -  
\sum_{i\in\N } \frac{4}{(x-s_i)^2}
.\end{align*}
This implies \As{F1} and \As{F2} immediately. 
We thus complete the proof. 
\qed 

\bigskip 

In the rest of the paper we devote to the proof of Theorems \ref{l:23} and \ref{l:24}.

\section{Logarithmic derivative of random point fields. } \label{s:4}

Let $ \mu $ be a probability measure on $ \SSS $ with locally bounded 
$ n $-point correlation function $ \rho ^{n}$ for each $ n \in \N $. 
Let $ \muone $ be the measure defined by \eqref{:22y} with $ k=1 $. 
In this section we present a sufficient condition for the existence of 
the logarithmic derivative $ \dmu $ in $ \Llocp (\muone )$. 

Let $ \Sr = \{ x \in \SS \, ;\, |x| < \rrr \}$ and 
$\Sr ^{n}$ denote the $ n$-product of $ \Sr $. 
Here and after, $ \cdot ^n$ denotes the $ n$-product of the set $ \cdot $. 
Let $ \{ \muN \} $ be a sequence of probability measures on $ \SSS $. 
We assume that their $ n $-point correlation functions $ \{\rho ^{\n ,n}\} $ 
satisfy for each $ r\in \N $ 
\begin{align} \label{:40a}&
\limi{\n } \rho ^{\n ,n} (\mathbf{x})= 
\rho ^{n} (\mathbf{x}) \quad \text{ uniformly on $\Sr ^{n}$}  
,\\\label{:40b}&
\sup_{ \n \in\N } \sup_{\mathbf{x}\in\Sr ^{n}} \rho ^{\n ,n} (\mathbf{x}) \le 
\cref{;40b} ^{-n} n ^{\cref{;40c}n}
,\end{align}
where $ 0 < \Ct \label{;40b}(r) < \infty $ and 
$ 0 < \Ct \label{;40c}(r)< 1 $ are constants independent of $ n \in \N $. 

Let $\map{\g }{\SS ^{2} }{\R }$ be measurable functions. 
For $ (x,\mathsf{y}) \in \SS \times \SSS $ and $ s > 0 $ we set 
\begin{align} & \label{:40c}
\ggNs (x,\mathsf{y}) = \sum_{|x-y_i|< s }\gN (x,y_i)
,\quad 
\rrNs (x,\mathsf{y}) =  \sum_{s \le |x-y_i|}\gN (x,y_i) 
,\end{align}
where $ \mathsf{y}=\sum_{i}\delta_{y_i}$. 
As for $ \rrNs $, we define only for $ \mathsf{y}$ such that 
$ \mathsf{y} (\SS ) < \infty $ in order to make the sum 
$\rrNs (x,\mathsf{y}) = \sum_{s \le |x-y_i|}\gN (x,y_i)$ finite. 
We note that $ \ggNs + \rrNs $ are independent of $ s $. 

Let $ \map{\uu , \uN }{\SS }{\R }$ and $ 1 < \hat{p} < \infty $. 
Assume that $  \muN $ has 
a logarithmic derivative $ \dmuN $ for each $ \n $ satisfying the following. 
\begin{align}\label{:40d}&
\dmuN (x,\mathsf{y})= \uN (x) + \ggNs (x,\mathsf{y}) + \rrNs (x,\mathsf{y})
,\\\label{:40e}&
\limi{\n }\uN =\uu \quad \text{ in } L^{\hat{p}}_{\mathrm{loc}}(\SS ,dx)
,\\\label{:40f}&
\limi{s}\limsupi{\n } 
  \int_{\Sr \ts \SSS } |\rrNs (x,\mathsf{y}) |^{\hat{p}}  d\muNone = 0 
.\end{align}

We quote: 
\begin{thm}[{\cite[Theorem 45]{o.isde}}] \label{l:41} 
Let $ 1 < p < \hat{p} $. 
Assume \eqref{:40a}--\eqref{:40f}. 
Then the logarithmic derivative $\dmu $ exists in $ \Llocp (\muone )$ and is given by 
\begin{align}\label{:41a}&
\dmu (x,\mathsf{y})= \uu (x) + \limi{s} \ggs (x,\mathsf{y})   
.\end{align}
The convergence $ \lim \ggs $ takes place in $ \Llocp (\muone )$. 
\end{thm}
\begin{rem}\label{r:41}
\tref{l:41} is a special case of \cite[Theorem 45]{o.isde}. 
In \cite[Theorem 45]{o.isde} extra terms such as $ \vN $ and $ \www $ appeared. These terms are vanished here. However, this is not the case for 
the Ginibre random point field and the Airy random point fields. 
\end{rem}

In practice, to check the condition \eqref{:40f} is the most hard part of the proof. 
So we quote a sufficient condition for this in terms of correlation functions. 
\begin{lem} \label{l:42}
We set  $ \SS _{s\infty}^{x} = \{ y \in \SS ;s \le |x-y| < \infty \}$. 
Let $ \rN _{x}$ be the $ n$-point correlation function 
of the reduced Palm measure $ \muNx $. 
Then \eqref{:40f} with $ \hat{p} = 2 $ follows from the following: 
\begin{align}
\label{:42A}&
\2  
|\int_{\SS _{s\infty}^{x}} \gN (x,y)\rNone (y) dy | = 0
,\\ \label{:42B}&
\2  |\int_{\SS _{s\infty}^{x}}\gN (x,y)
\{ \rNone _{x} (y)-\rNone (y)\}dy | = 0
,\\\label{:42C}& %% 
\2  |\int_{\SS _{s\infty}^{x}}|\gN (x,y)|^{2} \rNone (y)dy 
\\ \notag &\quad \quad \quad \quad \quad \quad \quad \quad \quad \quad \quad 
- \int_{(\SS _{s\infty}^{x})^{2}} 
 \gN (x,y)\cdot \gN (x,z)\rNtwo (y,z) dydz |= 0 
,\\% 
\label{:42D}&
\2  
|\int_{\SS _{s\infty}^{x}}|\gN (x,y)|^{2} \{ \rNone _{x} (y)-\rNone (y)\} dy
\notag \\ & \quad \quad
-\int_{(\SS _{s\infty}^{x})^{2} }
\gN (x,y)\cdot \gN (x,z) \{ \rNtwo _{x} (y,z)-\rNtwo (y,z)\}dydz |= 0 
.\end{align}
\end{lem}
\aaaaa 
\lref{l:42} is a special case of \cite[Lemma 52]{o.isde}.  
\bbbbb

\section{Finite particle approximations and proof of \tref{l:23}} \label{s:5}
In this section, we prove \tref{l:23}. 
For this we use \tref{l:41}. 
We will check the assumptions \eqref{:40a}--\eqref{:40f} posed in 
\tref{l:41}. 

We begin by giving finite particle approximations for 
Bessel random point fields. 
Let $ \{ L^{[\alpha ]}_{n} \} $ denote the Laguerre polynomials. 
Then by definition 
\begin{align}\label{:51a}&
 L^{[\alpha ]}_n(x) =\sum^n_{m=0} 
(-1)^m \binom{n+\alpha }{n-m}\frac{x^m}{m!}
.\end{align}
The associated monic polynomials $ \{ p^{[\alpha ]}_{n} \} $ are given by  
\begin{align}\label{:51aa}&
p^{[\alpha ]}_n (x) = (-1)^n n! L^{[\alpha ]}_n(x) 
= (-1)^n  \Gamma( n +1) L^{[\alpha ]}_n(x) 
.\end{align}
Let $ \w (x) = x^{\alpha }e^{-x}$. 
Then it is known \cite[301p, 302p]{sansone} that for 
$ m,n \in \{ 0 \}\cup \mathbb{N} $
\begin{align}\label{:51b}&
 \int_{0}^{\infty}
 L^{[\alpha ]}_m (x)  L^{[\alpha ]}_n (x) \, \w (x) dx = 
  \delta_{m,n}
 \frac{\Gamma( n+\alpha +1)}{\Gamma(n+1)}
 .\end{align}
From \eqref{:51aa} and \eqref{:51b} we immediately deduce that 
\begin{align}\label{:51B}&
 \int_{0}^{\infty} 
 p^{[\alpha ]}_{\n -1} (x) ^2  \, \w (x) dx = 
 \Gamma( \n +\alpha ) \Gamma( \n )
.\end{align}

Let $ \map{\kan }{(0,\infty)^2}{\mathbb{R}}$ be such that 
\begin{align}\label{:51C}&
 \kan (x,y)= 
 \sqrt{\w (x)\w (y)} \sum ^{\n -1}_{m = 0} 
 \frac{p^{[\alpha ]}_m  (x) p^{[\alpha ]}_m  (y) }
 {\int_{0}^{\infty} p^{[\alpha ]}_m  (z) ^2 \w (z)dz  }
.\end{align}
Then we deduce from 
the Christoffel-Darboux formula \cite[Proposition 5.1.3.]{forrester} 
that for $ x\not=y$ 
\begin{align}\label{:51D}&
\kan (x,y) =   
\frac{\sqrt{\w (x)\w (y)} }
{\int_{0}^{\infty} p^{[\alpha ]}_{\n -1}  (z) ^2 \w (z)dz }
 \frac
{ p^{[\alpha ]}_{\n }(x) p^{[\alpha ]}_{\n -1}  (y) - 
  p^{[\alpha ]}_{\n -1}(x) p^{[\alpha ]}_{\n }  (y)  } 
 {x-y}
.\end{align}
From \eqref{:51aa}--\eqref{:51B} combined with 
 a straightforward calculation, we obtain that 
\begin{align}\label{:51Zz}&
\kan (x,y) =  \sqrt{\w (x)\w (y)} 
 \frac{\Gamma (\n +1)} {\Gamma (\n + \alpha )} 
 \frac
{  L^{[\alpha ]}_{\n -1}(x) L^{[\alpha ]}_{\n }  (y) 
 - 
 L^{[\alpha ]}_{\n }(x)   L^{[\alpha ]}_{\n -1}  (y) } 
 {x-y}
.\end{align}

We now introduce the rescaled kernels $ \Kan $ as follows. 
\begin{align}\label{:51d}&
\Kan (x,y) = \frac{1}{4\n } \kan (\frac{x}{4\n },\frac{y}{4\n })
.\end{align}
Let $ \mubN $ be the determinantal random point field over $ (0,\infty)$ 
generated by $ (\Kan , dx)$. 
Then by construction the $ n$-point correlation functions of $ \mubN $ are given by 
\begin{align}\label{:51e}&
\rbN (x_1,\ldots,x_n) = \det [\Kan (x_i,x_j)]_{i,j=1,\ldots,n}
.\end{align}
It is known that $ \mubN (\sss (\SS )= \n )= 1 $, 
and so the labeled density function of 
$ \mubN $ is $ \rbNN $ up to the normalizing constant. Moreover, by the standard theory of the random matrix, we obtain, because of 
the calculation of  Vandermond determinant, 
\begin{align}\label{:51f}&
\rbNN (x_1,\ldots,x_{\n }) = 
c_{\n } e^{-\sum_{i=1}^{\n }x_i/4\n  } \prod_{j=1}^{\n }x_j^{\alpha }
\prod_{k<l}^{\n } |x_k-x_l|^{2 }
\end{align}
with the normalizing constant $ c_{\n }$ \cite{forrester}. 
We easily deduce from \eqref{:51f} that the logarithmic derivative 
$ \dlog ^{\mubN } $ of $ \mubN $ are given by 
\begin{align}\label{:51h}&
\dlog ^{\mubN }(x,\mathsf{y}) = 
-\frac{1}{4\n } + \frac{\alpha }{x} + \sum_{i=1}^{\n -1} \frac{2 }{x-y_i}
.\end{align}

\begin{lem} \label{l:51} 
 Let $ 1 \le \alpha < \infty $. 
 Then $ \{ \mubN \} $ satisfy \eqref{:40a}--\eqref{:40e} with 
\begin{align}\label{:51p}&
u^{\n } (x) = -\frac{1}{4\n } + \frac{\alpha }{x} , \quad 
g(x,y)= \frac{2}{x-y}
.\end{align}
\end{lem}
\aaaaa 
It is known (\cite[290p]{forrester}) that, 
for each $ (x,y)\in (0,\infty)^2 $, 
\begin{align}\label{:51g}&
\limi{\n } \Kan (x,y) = 
\Ka (x,y) 
.\end{align}
Furthermore, one can easily see that 
the convergence takes place compact uniformly in 
$ (x,y)\in [0,\infty)^2 $. 
We deduce \eqref{:40a} from \eqref{:51e} and \eqref{:51g} immediately. 

The condition \eqref{:40b} follows from \eqref{:51e} and \eqref{:51g}. 
In fact, from \eqref{:51g} and the definition \eqref{:51d} of the kernel 
$ \Kan $, we deduce that the norm 
$ \mathsf{k}_{\alpha ,i}^{\n ,n } (x_1,\ldots,x_n) $ 
of the $ i$th row vector of the matrix 
$ [\Kan (x_i,x_j)]_{i,j=1,\ldots,n} $ 
satisfies the following inequality. 
\begin{align}\label{:51H}& 
\sup_{(x_1,\ldots,x_n) \in S_r^n} 
\mathsf{k}_{\alpha ,i}^{\n ,n } (x_1,\ldots,x_n) \le \cref{;51} n^{1/2} 
.\end{align}
Here $ \Ct \label{;51} = \cref{;51} (r) $ is a positive constant 
independent of $ \n $ and $ n $. 
Hence from Hadamard's inequality we deduce that 
\begin{align}\label{:51i}& 
 \sup_{(x_1,\ldots,x_n) \in S_r^n} 
 |\det  [\Kan (x_i,x_j)]_{i,j=1,\ldots,n} | \le \cref{;51}^n n^{n/2}
.\end{align}

The conditions \eqref{:40c}--\eqref{:40e} 
are obvious from construction and \eqref{:51h}.  
\bbbbb

\bigskip

By \lref{l:51}, it only remains to prove \eqref{:40f}. 
Taking \lref{l:42} into account, we will deduce \eqref{:40f} from 
\eqref{:42A}--\eqref{:42D}. 
The key point of this is the estimate \eqref{:52a} in \lref{l:52}, 
which control the 1-point correlation functions $ \rbNone $ of $ \mub $. 
To prove \eqref{:52a} we prepare a bound of $ \rbNone $. 

\begin{lem} \label{l:50} 
Let $ \alpha > -1 $ and $ \omega > 1 $. Then for all $ \n \in \N $ 
\begin{align}\label{:52o}&
\rbNone (x) \le \frac{\cref{;b52f} }{\sqrt{x}} 
\quad \text{ for } 1 \le  x \le 4\n \omega 
.\end{align}
Here $ \Ct = \cref{;b52f} (\alpha , \omega ) \label{;b52f}$ 
 is a positive constant independent of $ x $ and $ \n $. 
\end{lem}

This lemma follows from an asymptotic formula of Hilb's type 
from \cite[Theorem 8.22.4, 199 p.]{szego}. 
Since the proof is long, although straightforward, 
we postpone it in Appendix (\sref{s:7}).

The next result is the most significant step of the proof. 
\begin{lem} \label{l:52}
The condition \eqref{:42A} is satisfied. 
Furthermore, it holds that  
\begin{align}\label{:52a}&
\2  
\int_{\SS _{s\infty}^{x}} \frac{1}{|x-y|}\rbNone (y) dy  = 0
.\end{align}
\end{lem}
\aaaaa 
Since \eqref{:42A} follows from \eqref{:52a}, we only prove \eqref{:52a} .  
We divide $ \SS _{s\infty}^{x} $ into 
two parts $ \SS _{s\infty}^{x}\cap [s,\omega \n ]$ and 
$ \SS _{s\infty}^{x}\cap [\omega \n , \infty )$, 
where $ \omega $ is a positive constant. 

We begin by the first case $ \SS _{s\infty}^{x}\cap [s,\omega \n ]$. 
From \eqref{:52o} we deduce that 
\begin{align} \label{:52g}& 
  \sup _{\n \in \mathbb{N}} \sup_{x\in\Sq }
\int_{\SS _{s\infty}^{x}\cap [s,\omega \n ]} \frac{\rbNone (y)}{|x-y|} dy 
\le 
 \sup_{x\in\Sq }
\int_{\SS _{s\infty}^{x}} 
\frac{\cref{;b52f}}{|x-y|\sqrt{|y|}} dy 
 \to 0
\end{align}
as $ s \to \infty $. 
 As for the second case $ \SS _{s\infty}^{x}\cap [\omega \n , \infty )$, 
we see that for $ r < \omega \n $
\begin{align}\label{:52q} & 
\sup_{x \in S_r} 
   \int_{\SS _{s\infty}^{x}\cap [\omega \n , \infty )} 
   \frac{\rbNone (y)}{|x-y|}  dy  
\leq  
  \frac{1}{ \n \omega - r } \int_{0}^{\infty} \rbNone (y) dy 
 =   \frac{\n }{ \n \omega - r }  
.\end{align}
Then we deduce from \eqref{:52q} that 
\begin{align}\label{:52h}
& \lim_{s \to \infty} \limsup_{\n \to \infty} \sup_{x \in S_r} 
   \int_{\SS _{s\infty}^{x}\cap [\omega \n , \infty )} 
   \frac{\rbNone (y)}{|x-y|}  dy  
\le  \frac{1}{\omega }  
.\end{align}
Combining \eqref{:52g} and \eqref{:52h}, we deduce that 
\begin{align} \label{:52z}& 
\2  
 \int_{\SS _{s\infty}^{x}}
  \frac{\rbNone (y)}{|x-y|} dy \le  \frac{1}{\omega }  
.\end{align}
Taking $ \omega > 0 $ to be arbitrary large in \eqref{:52z} 
yields \eqref{:52a}.  
\bbbbb

\bigskip 

We next prepare two properties \eqref{:82k} and \eqref{:53a} 
of determinantal kernels. 
We will repeatedly use these in the sequel.  
Let $ \mubNx $ be the reduced Palm measure of $ \mubN $ conditioned at $ x $ 
and let $ \rbNxn $ be its $ n $-point correlation function as before. 
Then $ \mubNx $ has a  determinantal structure with kernel 
\begin{align}\label{:82k}&
\Kanx (y,z)= \Kan (y,z) - 
\frac{\Kan (y,x)\Kan (x,z) }{ \Kan (x,x)}
.\end{align}
This relation follows from a general theorem on determinantal random point fields \cite[Theorem 1.7]{shirai-t}. 
Applying the Schwarz inequality to \eqref{:51C}, we deduce from \eqref{:51d} 
that   
\begin{align}\label{:53a}&
|\Kan (x,y) |
 \le \sqrt{\Kan (x,x)} \sqrt{\Kan (y,y)}
 =   \sqrt{\rbNone (x)} \sqrt{\rbNone (y)} 
.\end{align}
Here the equality in \eqref{:53a} follows from \eqref{:51e} with $ n = 1$.

\begin{lem} \label{l:53} 
 The condition \eqref{:42B} is satisfied. 
Furthermore, it holds that 
\begin{align}\label{:53d}&
\2  
\int_{\SS _{s\infty}^{x}} \frac{| \rbNxone (y)- \rbNone (y) | }{|x-y|} dy  = 0
.\end{align}
\end{lem}
\aaaaa 
From \eqref{:10a}, \eqref{:82k}, and \eqref{:53a}, we deduce that 
\begin{align}\label{:82u}
|\rbNxone (y)- \rbNone (y) | &
 = | \frac{\Kan (y,x)\Kan (x,y)}{\Kan (x,x)} |
%\\ \label{:53c}& 
 \le \Kan (y,y) = \rbNone (y) 
.\end{align}
Hence \eqref{:53d} is immediate from \eqref{:52a}. 
The condition \eqref{:42B} follows from \eqref{:53d} immediately. 
\bbbbb

\begin{lem} \label{l:54} 
The condition \eqref{:42C} is satisfied. 
Furthermore, it holds that 
\begin{align}\label{:54a}&
\2  \int_{\SS _{s\infty}^{x}} \frac{\rbNone (y)}{|x-y|^{2}} dy 
+ \int_{(\SS _{s\infty}^{x})^{2}} \frac{ \rbNtwo (y,z) }{|x-y||x-z|}dydz = 0 
.\end{align}
\end{lem} 
\aaaaa 
From \lref{l:52} we easily deduce that 
\begin{align}\label{:54b}&
\2 \int_{\SS _{s\infty}^{x}} \frac{\rbNone (y)}{|x-y|^{2}} dy = 0 
.\end{align}
From \eqref{:51e} and \eqref{:53a} we see that 
\begin{align}\label{:54c}
 \rbNtwo (y,z) & = \rbNone (y)\rbNone (z) - \Kan (y,z)\Kan (z,y) 
 \le 2  \rbNone (y)\rbNone (z)
.\end{align}
Hence from \eqref{:54c} and Fubini's theorem, we deduce that 
\begin{align}\label{:54d} 
\int_{(\SS _{s\infty}^{x})^2} 
\frac{\rbNtwo (y,z)  }{|x-y||x-z|} dydz 
\le 
 &  \int_{(\SS _{s\infty}^{x})^2} 
\frac{ 2 \rbNone (y)\rbNone (z) }{|x-y||x-z|} dydz 
\\ \notag  \le \, 
 & 2 \, (\int_{\SS _{s\infty}^{x}} \frac{ \rbNone (y)}{|x-y|} dy)^2 
.\end{align}
Then from \eqref{:54d} and \lref{l:52} we deduce that 
\begin{align}\label{:54D}&
\2 \int_{(\SS _{s\infty}^{x})^2} 
\frac{\rbNtwo (y,z)  }{|x-y||x-z|} dydz = 0 
.\end{align}
From \eqref{:54b} and \eqref{:54D}, we conclude \eqref{:54a}.  
This implies \eqref{:42C}. 
\bbbbb

%%%%%%%%%%%%%%%%%%%%%%%%%%

\bigskip

\begin{lem} \label{l:55} 
The condition \eqref{:42D} is satisfied. Furthermore, it holds that 
\begin{align}\label{:55a}&
\2  
\int_{\SS _{s\infty}^{x}}\frac{| \rbNxone  (y)-\rbNone (y) |}{|x-y|^{2}} 
 dy
\\ & \quad \quad \notag 
 + \int_{(\SS _{s\infty}^{x})^{2} }
\frac{| \rbNxtwo  (y,z)-\rbNtwo (y,z) |}{|x-y||x-z|}   dydz = 0 
.\end{align}\end{lem} 
\aaaaa 
We deduce from \lref{l:53} that 
\begin{align}\label{:55b}&
\2 \int_{\SS _{s\infty}^{x}} 
 \frac{| \rNone _{x} (y)-\rNone (y)|}{|x-y|^{2}} dy  = 0 
.\end{align}
To estimate the second term of \eqref{:55a} we observe that 
\begin{align}\label{:55c}& 
\rbNxtwo (y,z)- \rbNtwo (y,z)
\\ \notag 
= &
- \Kan (y,x)\Kan (x,y) \frac{\Kan (z,z)}{\Kan (x,x)}
- \Kan (z,x)\Kan (x,z)  \frac{\Kan (y,y)}{\Kan (x,x)}
\\ \notag &
+ \Kan  (y,x)\Kan (x,z)\Kan (z,y) \frac{1}{\Kan (x,x)}
+ \Kan (y,z)\Kan (z,x)\Kan (x,y) \frac{1}{\Kan (x,x)}
.\end{align}
Then applying \eqref{:53a} to each term of the right-hand side, 
we obtain 
\begin{align}\label{:55d}&
|\rbNxtwo (y,z)- \rbNtwo (y,z)| 
% \\ \notag &
% \le \Kan (y,y)\Kan (x,x) \frac{\Kan (z,z)}{\Kan (x,x)}
% +  \Kan (z,z)\Kan (x,x) \frac{\Kan (y,y)}{\Kan (x,x)}
% \\ \notag &
% +\Kan (y,y)\Kan (z,z)\Kan (x,x)\frac{1}{\Kan (x,x)}
% +\Kan (y,y)\Kan (z,z)\Kan (x,x)\frac{1}{\Kan (x,x)}
% \\ \notag &
%= 4\Kan (y,y)\Kan (z,z) = 
\le 4\rbNone (y)\rbNone (z)
.\end{align}
Hence from \eqref{:55d} and Fubini's theorem, we deduce that 
\begin{align}\label{:55e}   
\int_{(\SS _{s\infty}^{x})^{2} }
\frac{ | \rbNxtwo  (y,z)-\rbNtwo (y,z) | }{|x-y||x-z|}dydz 
\le & \int_{(\SS _{s\infty}^{x})^{2} }
\frac{4\rbNone (y)\rbNone (z)}{|x-y||x-z|} dydz 
\\ \notag
=  & \, 4 \, 
\left\{ 
\int_{\SS _{s\infty}^{x}}\frac{\rbNone (y)}{|x-y|} dy \right\}^2 
.\end{align}
Then from \eqref{:55e} and \lref{l:52}, we see that 
\begin{align}\label{:55E}&
\2 \int_{(\SS _{s\infty}^{x})^{2} }
\frac{ | \rbNxtwo  (y,z)-\rbNtwo (y,z) | }{|x-y||x-z|}dydz = 0 
.\end{align}
From \eqref{:55b} and \eqref{:55E}, we obtain \eqref{:55a}. 
This implies \eqref{:42D}. 
\bbbbb 

\bigskip 

%\medskip
\noindent 
{\em Proof of Theorems \ref{l:23}. } 
From \lref{l:51} we see that $ \mub $ satisfy \eqref{:40a}--\eqref{:40e}. 
From \lref{l:52}--\lref{l:55}, 
we deduce that $ \mub $ satisfy \eqref{:42A}--\eqref{:42D}.   
This combined with \lref{l:42} yields \eqref{:40f}. 
We thus see that all the conditions \eqref{:40a}--\eqref{:40f} of 
\tref{l:41} are fulfilled. 
Hence from \tref{l:41}, 
we obtain \tref{l:23} with the logarithmic derivative 
$ \dlog $ given by \eqref{:41a}. 
\qed

\section{Proof of \tref{l:24}}\label{s:6} 

In this section we prove \tref{l:24}. 
For this we use \cite[Theorem 2.2]{o.rm2}. 
We prepare a result from \cite{o.rm2}, which is a special case of 
\cite[Theorem 2.2]{o.rm2}. 

In the next theorem, we take $ \SS = \ZI $ or $ \SS = \mathbb{R}$. 
Let $ \mu $ be a random point field on $ \SS $. 
We assume three conditions. 

\medskip

\noindent
\Ass{B.1} The random point field $ \mu $ has 
a locally bounded, $ n $-point correlation function 
$ \rho ^n $ for each $ n \in  \N    $.

\medskip 
\noindent 
\Ass{B.2} 
There exists a sequence of random point fields  
$\{ \muN \}_{\n \in \N }$ over $ \SS $ satisfying the following.  

\noindent 
\thetag{1} 
The $ n $-point correlation functions $  \rN $ of $  \muN $ satisfy  
\begin{align} \label{:61a} &
\lim _{\n \to \infty } \rN (\mathbf{x}_n) = \rho ^n (\mathbf{x}_n) 
\quad \text{ a.e.} 
\quad \text{ for all $ n \in  \N    $,}
\\ \label{:61b} 
& \sup 
\{  \rN (\mathbf{x}_n)  ; \n \in \N     ,\, \mathbf{x}_n \in  \Sr ^n \} 
\le \{ \cref{;70} n ^{\cref{;62} }\} ^n
\quad \text{ for all $ n,r\in  \N    $} 
,\end{align}
where $ \mathbf{x}_n = (x_1,\ldots,x_n) \in S ^{n}$, 
$ \Ct \label{;70}=\cref{;70}(r) >0$, and $ \Ct \label{;62} = \cref{;62} (r) < 1 $ 
are constants depending on $ r \in  \N    $. 

\noindent 
\thetag{2} \  $ \muN (\mathsf{s}( S ) = \nN ) = 1 $ for each $ \n $, 
where $ \nN \in  \N  $ are strictly increasing. 

\noindent \thetag{3} 
$ \muN $ is a $ (\PhiN ,\PsiN ) $-canonical Gibbs measure 
for each $ \n $. 

\noindent \thetag{4} 
$ \PhiN $ satisfy the following. 
\begin{align}\label{:61c}&
\limi{\n } \PhiN (x) = \Phi (x) \text{ for a.e.\ \!\! $ x $,} 
\quad  
%\\ \notag &
\infN \inf _{x \in  S } \PhiN (x)   > -\infty 
.\end{align}

Let $ \mathsf{x} =\sum _{i}\delta_{x _i }$. 
For $1 \le r < s \le \infty $ let 
$ \map{\mathsf{v}_{\ell ,rs} }{\SSS }{\mathbb{R}}$ such that 
\begin{align}\label{:62b}&
\mathsf{v}_{\ell ,rs} (\mathsf{x} )= \beta 
\big\{\sum _{x _i \in \SS _s \backslash \SS _r }
 \frac{1}{{ x }_i^{\ell } } \big\} \quad (\ell \ge 1) 
.\end{align}
Note that the sum in \eqref{:62b} makes sense for $ \muN $-a.s. $ \mathsf{x} $ 
even if $ s=\infty $. 
Indeed, by \thetag{2} of \Ass{B.2}, the total number of particles is $ \nN $ under $ \muN $. 
Hence, $ \mathsf{v}_{\ell ,rs} (\mathsf{x} )$ is well defined and finite for 
$ \muN $-a.s. $ \mathsf{x} $, for all $ \n \in  \N    $.

\smallskip 
\noindent 
\Ass{B.3} There exists an $ \ell _{0}$ such that $ \ell _{0}\in \N $ 
and that 
\begin{align} \label{:62d}&
\supN  \{ \int_{1\le |x |<\infty } 
 \frac{1 }{\ |  x |^{\ell _{0}}} \, \rNone (x )dx \} < \infty 
\\\intertext{and that, for each $  1 \le \ell < \ell _{0}$, }
\label{:62f} &
\limi{s} \supN  \|  \sup _{\n \in \N } \mathsf{v}_{\ell ,s\infty }
 \,  \|_{L^1(\SSS , \mu ^{\n })} = 0 
.\end{align}
When $ \ell _{0} = 1 $, we interpret that \eqref{:62f} always holds.  
The following is a special case of \cite[Theorem 2.2]{o.rm2}. 
We remark that the assumptions \Ass{B.1}, \Ass{B.2} and \Ass{B.3} 
correspond to 
\thetag{H.1}, \thetag{H.2} and \thetag{H.4} in \cite[Theorem 2.2]{o.rm2}, 
respectively. 

\begin{thm}[{\cite[Theorem 2.2]{o.rm2}}]\label{l:62} 
Assume \Ass{B.1}, \Ass{B.2} and \Ass{B.3}. 
Then $ \mu $ is a $ (\Phi , \Psi ) $-quasi-Gibbs measure. 
\end{thm}

\noindent 
{\em Proof of \tref{l:24}. }
We check the assumptions \Ass{B.1}, and \Ass{B.2}, and 
\Ass{B.3} in \tref{l:62}. 
We take $ \mu = \mub $ and $ \mu ^{\n }=  \mubN $. 
Then \Ass{B.1} and \Ass{B.2} are satisfied. 
Furthermore, we take $ \ell _{0} = 1 $. 
Then from \eqref{:10b} and \lref{l:52}, we deduce \eqref{:62d} easily. 
Hence we conclude \tref{l:24} from \tref{l:62}. 
\qed

\section{Appendix: Proof of \lref{l:50}. }\label{s:7}
In this section we prove \lref{l:50}. 
Let $ -1 < \alpha < \infty $.  Let $ L^{[\alpha ]}_{\n } $
denote the Laguerre polynomial and 
$ \w (x) = e^{- x} x^{\alpha }$ as before. Let 
\begin{align}\label{:71z}&
 \Ma  (x) = 
\{   \mathsf{w} _{\alpha +1}^{\frac{1}{2}}  
L^{[\alpha +1]}_{\n -1} \, 
      \w ^{\frac{1}{2}} L^{[\alpha ]}_{\n -1} 
   -  \w ^{\frac{1}{2}}L^{[\alpha ]}_{\n } \, 
   \mathsf{w} _{\alpha +1}^{\frac{1}{2}} 
    L^{[\alpha +1]}_{\n -2}  \, 
   \} (x)
.\end{align}
Then from a straightforward calculation we obtain the following.  
\begin{lem} \label{l:71} 
There exists a positive constant $ \Ct \label{;72}$ such that  
\begin{align}\label{:71a}& 
 \rbNone (y) \le  
 \frac{\cref{;72}}{\sqrt{y}} \, 
 \frac{1}{\n ^{\alpha - \frac{1}{2}}} 
 \Ma   (\frac{y}{4\n })
\end{align}
for all $ \n \in \mathbb{N}$ and $ y \in (0,\infty)$. 
\end{lem}
\begin{proof}
From \eqref{:51d} and \eqref{:51e}, we see that 
\begin{align}\label{:71b}&
\rbNone (y) = \Kan (y,y) = \frac{1}{4\n  } 
\kan  (\frac{y}{4\n },\frac{y}{4\n }) 
.\end{align}
Hence we will estimate $ \kan (x,x) $. 
Taking $ y \to x $ in \eqref{:51Zz}, we deduce that 
\begin{align}\label{:71c}
\kan (x,x) & 
= \w (x)  \frac{\Gamma (\n +1)} {\Gamma (\n + \alpha )} 
\{ - 
\frac{d L^{[\alpha ]}_{\n } }{dx} L^{[\alpha ]}_{\n -1} 
  +
L^{[\alpha ]}_{\n }  \frac{d L^{[\alpha ]}_{\n -1} }{dx} 
    \, 
   \} (x)  
\\ \notag 
& = 
\w (x)  \frac{\Gamma (\n +1)} {\Gamma (\n + \alpha )} 
\{    L^{[\alpha +1]}_{\n -1}  L^{[\alpha ]}_{\n -1} 
   -  L^{[\alpha ]}_{\n }  L^{[\alpha +1]}_{\n -2}  \, 
   \} (x)  
 \\ \notag & = 
 \frac{\w (x)^{\frac{1}{2}}}{\mathsf{w}_{\alpha +1}(x)^{\frac{1}{2}}}   
  \frac{\Gamma (\n +1)} {\Gamma (\n + \alpha )} 
\{    \mathsf{w} _{\alpha +1}^{\frac{1}{2}}L^{[\alpha +1]}_{\n -1} 
      \w ^{\frac{1}{2}} L^{[\alpha ]}_{\n -1} 
   -  \w ^{\frac{1}{2}}L^{[\alpha ]}_{\n } 
      \mathsf{w} _{\alpha +1}^{\frac{1}{2}} L^{[\alpha +1]}_{\n -2}  \, 
   \} (x)
\\ \notag 
& = 
  \frac{1}{\sqrt{x}}   
  \frac{\Gamma (\n +1)} {\Gamma (\n + \alpha )} 
 \Ma   (x)
.\end{align}
Here we used the formula 
$ \frac{d L^{[\alpha ]}_{\n } }{dx}= -  L^{[\alpha +1]}_{\n -1} $ 
(see \cite[102 p]{szego}) for the second line, and 
\eqref{:71z} for the last line. 
Taking $ x = \frac{y}{4\n }$ in \eqref{:71c}, we obtain that 
\begin{align}\label{:71d}&
% \Kan (y,y) = 
\frac{1}{4\n }\kan (\frac{y}{4\n },\frac{y}{4\n })
= 
\frac{1}{\sqrt{4\n y}} 
  \frac{\Gamma (\n +1)} {\Gamma (\n + \alpha )} 
\Ma  (\frac{y}{4\n })
.\end{align}
Clearly, there exists a positive constant $ \cref{;72} $ such that 
\begin{align}\label{:71e}&
\frac{1}{\sqrt{4\n }}   \frac{\Gamma (\n +1)} {\Gamma (\n + \alpha )} 
\le \cref{;72} \n ^{-\alpha + \frac{1}{2}} 
\quad \text{ for all } \n \in \mathbb{N}
.\end{align}
From \eqref{:71d} and \eqref{:71e} we obtain \eqref{:71a}.  
\end{proof}

From \lref{l:71}, our next task is to prove 
\begin{align}\label{:72z}&
 \Ma   (\frac{y}{4\n }) = O (\n^{\alpha -\frac{1}{2}}) 
.\end{align}
Here the bound $ O (\n^{\alpha -\frac{1}{2}})$ is taken 
to be uniform in 
$ \cref{;Hil} \n ^{-1} \le x \le \omega $. 
For this we quote an asymptotic formula of Hilb's type from \cite{szego}. 
\begin{lem}[{\cite[Theorem 8.22.4, 199 p.]{szego}}] \label{l:72}
Let $ -1 < \alpha < \infty  $. 
Let $ \Ct \label{;Hil},\, \omega > 0 $ be fixed. 
Then each Laguerre polynomial $ L^{[\alpha ]}_{\n } $ satisfies, 
for all $ \cref{;Hil} \n ^{-1} \le x \le \omega $, 
\begin{align}\label{:72a}& 
 \w (x)^{\frac{1}{2}} 
 L^{[\alpha ]}_{\n } (x) = 
\Ana  
% \frac{\Gamma (\n + \alpha + 1)}{\Gamma (\n +1) }
J_{\alpha } ( \sqrt{4Nx} ) 
+ x^{\frac{5}{4}} O (\n ^{\alpha /2 -\frac{3}{4}})
.\end{align}
Here $ N = \n + (\alpha + 1)/2$ and the bound 
$ O (\n ^{\alpha /2 -\frac{3}{4}})$ holds uniformly in 
$ \cref{;Hil} \n ^{-1} \le x \le \omega $. 
Furthermore, $ \Ana $ is defined by 
\begin{align}\label{:72b}&
 \Ana  = 
 \frac{\Gamma (\n + 1+ \alpha )}{ N ^{\alpha / 2} \Gamma (\n + 1 )}
.\end{align}
\end{lem}
\begin{proof}
This lemma follows from \thetag{8.22.4} and \thetag{8.22.5} in 
Szeg\"{o} \cite[Theorem 8.22.4, 199 p.]{szego} 
immediately. We remark that we use the first equation 
of \thetag{8.22.5} in 
Szeg\"{o} \cite[Theorem 8.22.4, 199 p.]{szego} as well as \thetag{8.22.4} 
in \cite{szego}. 
\end{proof}

\medskip

From \lref{l:72} we see the following.  
\begin{lem} \label{l:73} 
For all $ \n \in \mathbb{N }$ and $ \cref{;Hil} \n ^{-1} \le x \le \omega $ 
\begin{align}\label{:73a}&
 \Ma  (\frac{y}{4\n }) =   O(\n^{\alpha + \frac{1}{2}}) 
   \\ \notag  & \quad 
\cdot 
  \left[   
   \left\{ J_{\alpha }((\frac{ N -1}{\n }y)^{\frac{1}{2}}) + 
(\frac{y}{4\n })^{\frac{5}{4}} O(\frac{1}{\n ^{\frac{3}{4}}}) \right\}
 \left\{ J_{\alpha +1}((\frac{ N -1}{\n }y)^{\frac{1}{2}}) + 
(\frac{y}{4\n })^{\frac{5}{4}} O(\frac{1}{\n ^{\frac{3}{4}}}) \right\}
 \right.
\\ \notag & \quad \quad  \left.
- 
 \left\{ J_{\alpha +1}((\frac{ N -2}{\n }y)^{\frac{1}{2}}) + 
(\frac{y}{4\n })^{\frac{5}{4}} O(\frac{1}{\n ^{\frac{3}{4}}}) 
\right\}
\left \{ J_{\alpha }((\frac{ N }{\n }y)^{\frac{1}{2}}) + 
(\frac{y}{4\n })^{\frac{5}{4}} O(\frac{1}{\n ^{\frac{3}{4}}}) 
\right\} 
   \right] 
.\end{align}
\end{lem}
\begin{proof} 

From \eqref{:72a} we easily deduce that, 
for all $ \cref{;Hil}\n ^{-1} \le x \le \omega $, 
\begin{align}\label{:91Q}
 \w (x)^{\frac{1}{2}}  L^{[\alpha ]}_{\n } (x) 
 & = 
\Ana  
\left\{ 
J_{\alpha } ( \sqrt{4Nx} ) 
+ x^{\frac{5}{4}} O(\frac{1}{\n ^{\frac{3}{4}}})
\right\}
.\end{align}
Then taking $ x = y /4\n $ in \eqref{:91Q} we deduce that  
\begin{align} \label{:91q} 
\w (\frac{y}{4\n })^{\frac{1}{2}} 
 & L^{[\alpha ]}_{\n } (\frac{y}{4\n }) 
 =  
 \Ana 
 \left\{ J_{\alpha }((\frac{ N }{\n }y)^{\frac{1}{2}}) + 
(\frac{y}{4\n })^{\frac{5}{4}} O(\frac{1}{\n ^{\frac{3}{4}}}) 
\right\}
.\end{align}
A simple calculation shows that 
\begin{align}\label{:91R}&
\Ana = O(\n^{\frac{1}{2}\alpha }) 
\end{align}
Hence we deduce \eqref{:73a} from \eqref{:71z}, \eqref{:91q} 
and \eqref{:91R} immediately. 
\end{proof}

\noindent 
{\em Proof of \lref{l:50}. } 
We will calculate the right-hand side of \eqref{:73a}.

Recall that $ N = \n + \frac{1}{2} (\alpha + 1) $. Then 
$ \frac{N}{\n }- \frac{N-1}{\n }= \frac{1}{\n } $. 
Hence from Taylor expansion of 
$ J_{\alpha }(\sqrt{y}) $ 
 and 
$ J_{\alpha +1}(\sqrt{y}) $, 
 we deduce that 
\begin{align}\label{:991}&
 J_{\alpha }((\frac{ N -1}{\n }y)^{\frac{1}{2}}) 
 J_{\alpha +1}((\frac{ N -1}{\n }y)^{\frac{1}{2}}) 
- 
J_{\alpha }((\frac{ N }{\n }y)^{\frac{1}{2}} ) 
J_{\alpha +1}((\frac{ N -2}{\n }y)^{\frac{1}{2}} ) 
\\ \notag & = 
\frac{1}{\n } \left[
 J_{\alpha }((\frac{ N -1}{\n }y)^{\frac{1}{2}} ) 
\, \n 
\left\{ 
 J_{\alpha +1}((\frac{ N -1}{\n }y)^{\frac{1}{2}} ) 
- 
 J_{\alpha +1}((\frac{ N -2}{\n }y)^{\frac{1}{2}}) \right\}
  \right.
  \\ 
&\left. \quad \quad \quad - 
\n \left\{
J_{\alpha }((\frac{ N }{\n }y)^{\frac{1}{2}} ) 
- 
 J_{\alpha }((\frac{ N -1}{\n }y)^{\frac{1}{2}}) \right\}
 J_{\alpha +1}((\frac{ N -2}{\n }y)^{\frac{1}{2}}) 
 \right]
\\ \notag 
& = O(\frac{1}{\n }) \left[
\{ \frac{d}{dy } 
  J_{\alpha }(\sqrt{y})  \} \cdot  J_{\alpha +1}(\sqrt{y}) 
 - 
J_{\alpha }(\sqrt{y}) 
\{ \frac{d}{dy }  J_{\alpha +1}(\sqrt{y}) \}  \right]
\\ \notag 
& =  O (\frac{1}{\n } )
.\end{align}
We also see that 
\begin{align}\label{:992}&
J_{\alpha }((\frac{ N }{\n }y)^{\frac{1}{2}}) 
 (\frac{y}{4\n })^{\frac{5}{4}} O(\frac{1}{\n ^{\frac{3}{4}}}) 
 = 
O (1) \frac{1}{y^{\frac{1}{4}}}
 (\frac{y}{4\n })^{\frac{5}{4}} 
  O(\frac{1}{\n ^{\frac{3}{4}}}) 
 = 
O (\frac{1}{\n } )
.\end{align}
Here we used $ |J_{\alpha } (t)| \le O(1) t^{-1/2}$ and 
$ \frac{y}{4\n } = O(1)$. 
Substituting \eqref{:991} and \eqref{:992} 
together with similar relations into \eqref{:73a}, we deduce that  
\begin{align}\label{:993}
 \Ma (\frac{y}{4\n })
 =  & O (\n ^{\alpha + \frac{1}{2}})  O (\frac{1}{\n}) 
 = O (\n ^{\alpha - \frac{1}{2}}) 
.\end{align}
This together with \lref{l:71} completes the proof of \lref{l:50}. 
\bbbbb

\section{Appendix 2: No particles hit the origin. } \label{s:8}
In this setion we prove that no partciles hit the origin if $ 1 \le \alpha $. 
For this it is enough to prove that the capacity of the set 
being at least one particle at the origin is zero. 
\begin{lem} \label{l:81} 
Let $ \mathsf{A} = \{ \mathsf{s}\in\mathsf{S};\, \mathsf{s}(\{ 0 \} ) \ge 1 \} $. 
Suppose $ \alpha \ge 1 $. Then 
\begin{align}\label{:81a}&
\mathrm{Cap}^{\mu }(\mathsf{A}) = 0
.\end{align}
Here $ \mathrm{Cap}^{\mu }$ is the capacity associated with the Dirichlet space 
$(\Eazero ,\d ^{\mu } ,\Lm )$ as before. 
\end{lem}
\begin{proof}
Set 
$ \mathcal{D}_{\circ ,r} = \{ f\in \di ; \text{$ f $ is $ \sigma [\pi _r]$-measurable} \} $, 
where $ \pi _r (\mathsf{s})= \mathsf{s} (\cdot \cap \SS _{b_r})$ and $ b_r $ are 
 as in \dref{dfn:2}. 
Let $ \d _r^{\mu } $ be the closure of $\mathcal{D}_{\circ ,r} $  
with respect to $ (\Eazero ,\d ^{\mu } ) $ on $ \Lm $. It is then clear that 
%% $ (\Eazero ,\d _r^{\mu } ) $ on $ \Lm $ are decreasing in $ r $ and satsfy 
\begin{align}\label{:82a}&
\d _r^{\mu } \subset \d ^{\mu }
.\end{align}
We can regard $ (\Eazero ,\d _r^{\mu } ) $ on $ \Lm $ 
as a quasi-regular Dirichlet form on $ L^2(\mathsf{\SS} _r, \mu _r)$, 
where $ \mathsf{\SS} _r$ is the configuration space over $ \SS _{b_r }$, and 
$ \mu _r = \mu \circ \pi _r^{-1}$ is regarded as a random point field on $ \SS _{b_r }$. 

Let $ \mathrm{Cap}_r^{\mu }$ be the capacity associated with the Dirichlet form 
$ (\Eazero ,\d _r^{\mu } ) $ on $ L^2(\mathsf{\SS} _r, \mu _r)$. 
We then obtain from \eqref{:82a} that for any $ r \in \N $ 
\begin{align}\label{:81b}&
\mathrm{Cap}^{\mu } (\mathsf{A})\le \mathrm{Cap} _r^{\mu }(\mathsf{A})
.\end{align}
Fix $ r \in \N $ and set $ \mathsf{A}_r^m = \mathsf{A}\cap \{ \mathsf{s}(\Sr )=m \} $. 
We then see that 
$ \mathsf{A} = \sum_{m=1}^{\infty}\mathsf{A}_r^m $. 
Hence we deduce from \eqref{:81b} and  the sub-additivity of capacity that 
\begin{align}\label{:81c}&
\mathrm{Cap} _r^{\mu }(\mathsf{A} ) 
\le \sum_{m=1}^{\infty} \mathsf{Cap} _r^{\mu } (\mathsf{A}_r^m )
.\end{align}
Taking \eqref{:81b} and \eqref{:81c} into account, we deduce \eqref{:81a} from 
\begin{align}\label{:81d}&
\mathrm{Cap} _r^{\mu }(\mathsf{A}_r^m ) = 0 
\quad \text{ for all } m \in \N 
.\end{align}

Let $ \mr ^m $ be the symmetric labeled density of $ \mu \circ \pi _r^{-1}$ on $ \Sr ^m $. Then 
%there exists a constant $ \Ct \label{;81e} $ such that 
\begin{align}\label{:81e}&
\mr ^m (x_1,\ldots,x_m) \le \rbm (x_1,\ldots,x_m) \quad \text{ on } 
\Sr ^m 
.\end{align}
Since $ \mu $ is a determinantal random point field with kernel \eqref{:10b}, 
we see that 
\begin{align}\label{:81f}&
|\Ka (x_i,x_j)| \le \{\Ka (x_i,x_i) \Ka (x_j,x_j)\}^{1/2}= 
\{\rbone (x_i) \rbone (x_j)\}^{1/2}
.\end{align}
We then obtain from \eqref{:10a} and \eqref{:81f} that 
\begin{align}\label{:81g}&
\rbm (x_1,\ldots,x_m)  \le m! \prod_{i=1}^m \rbone (x_i) 
.\end{align}
From \eqref{:81e} and \eqref{:81g} combined with the expression \eqref{:10bb}, 
there exists a positive constant $ \Ct \label{;81g}$ depending on $ r,m \in \N $ 
such that 
\begin{align}\label{:81h}&
\mr ^m (x_1,\ldots,x_m) \le \cref{;81g} \prod_{i=1}^m x_i^{\alpha } 
\quad \text{ for all } (x_1,\ldots,x_m)\in \Sr ^m 
.\end{align}
From \eqref{:81h} and a direct calculation using $ 1 \le \alpha $, 
we obtain \eqref{:81d}. This completes the proof of \eqref{:81a}.    
\end{proof}

%%%[]



%% References


\small{

}

% Author, Article title, Journal, Volume, page numbers (year)
% Author, Book title, page numbers. Publisher, place (year)
% etc

%% \bibliographystyle{model1a-num-names}
%% \bibliography{<your-bib-database>}






% \end{thebibliography}

\end{document}